\documentclass[a4paper, 11pt]{article}
\usepackage[utf8]{inputenc}     % Si l'encodage du fichier est unicode UTF-8
\usepackage[T1]{fontenc}
\usepackage{lmodern}
\usepackage{graphicx}
\usepackage[english]{babel}
\usepackage{subfigure}
\usepackage{amsmath, amsfonts, amssymb,amsthm}  
\usepackage{textcomp}
\usepackage{mathrsfs}
\usepackage{multirow}
\usepackage{float}

\theoremstyle{plain}
\newtheorem{theorem}{Theorem}[section]
\newtheorem{proposition}[theorem]{Proposition}
\newtheorem{lemma}[theorem]{Lemma}

\newtheorem{remark}[theorem]{Remark}

\theoremstyle{definition}
\newtheorem{definition}[]{Definition}

\newcommand{\bydef}{\,\stackrel{\mbox{\tiny\textnormal{\raisebox{0ex}[0ex][0ex]{def}}}}{=}\,}

\newcommand{\cL}{\mathcal{L}}
\newcommand{\ba}{{\bar a}}

\newcommand{\R}{\mathbb{R}}
\newcommand{\C}{\mathbb{C}}
\DeclareMathOperator{\diag}{diag}

\title{Computation of maximal local (un)stable manifold patches\\
by the parameterization method}

\author{Maxime Breden \thanks{CMLA, ENS Cachan, CNRS, Université Paris-Saclay, 61 avenue du Pr\'esident Wilson, 94230 Cachan, France \& Département de Mathématiques et de Statistique, Université Laval, 1045 avenue de la Médecine, Québec, QC, G1V 0A6, Canada. {\tt mbreden@ens-cachan.fr}}
\and Jean-Philippe Lessard\thanks {D\'epartement de Math\'ematiques et de Statistique, Universit\'e Laval, 1045 avenue de la M\'edecine, Qu\'ebec, QC, G1V0A6, Canada. {\tt jean-philippe.lessard@mat.ulaval.ca}}
\and Jason D. Mireles James \thanks{Florida Atlantic University
Science Building, Room 234, 777 Glades Road, Boca Raton, Florida, 33431, USA.
 {\tt jmirelesjames@fau.edu}}}

\date{}

\begin{document}

\maketitle

%\vspace{-.5cm}

\begin{abstract}
In this work we develop some automatic procedures for computing high order polynomial expansions of local (un)stable manifolds for equilibria of differential equations.  Our method incorporates validated truncation error bounds, and maximizes the size of the image of the polynomial approximation relative to some specified constraints.  More precisely we use that the manifold computations depend heavily on the scalings of the eigenvectors: indeed we study the precise effects of these scalings on the estimates which determine the validated error bounds.  This relationship between the eigenvector scalings and the error estimates plays a central role in our automatic procedures.  
In order to illustrate the utility of these methods we present several applications, including visualization of invariant manifolds in the Lorenz and FitzHugh-Nagumo systems and an automatic continuation scheme for (un)stable manifolds in a suspension bridge problem.  In the present work we treat explicitly the case where the 
eigenvalues satisfy a certain non-resonance condition.
\end{abstract}

\begin{center}
{\bf \small Key words.} 
{ \small Invariant manifold, parameterization method, radii polynomials, \\ algorithms, a posteriori analysis}
\end{center}

\begin{center}
\end{center}

\section{Introduction} \label{sec:intro}

Invariant sets are fundamental objects of study in 
dynamical systems theory.  Sometimes we are interested in an
invariant set which is a smooth manifold, and we seek a representation
of a chart patch as the graph of a function or as the image of a chart map.
Semi-numerical methods providing high order formal
expansions of invariant manifolds have a long history in dynamical 
systems theory.  We refer to the 
lecture notes of Sim\'o \cite{1990mmcm.conf..285S},
the historical remarks in Appendix B of the paper
by Cabr\'{e}, Fontich, and de la Llave \cite{MR2177465},
the manuscript of Haro \cite{alexADmanuscript}, as well as the book 
by Meyer and Hall \cite{MR2468466} for more complete 
discussion of this literature.

The present work is concerned with algorithms for 
computing local stable/unstable manifolds of equilibria 
solutions of differential equations, with validated error bounds.
The methods employed here have some free computational 
parameters and we are especially interested in 
choosing these in an automatic way.  
We employ the parameterization 
method of \cite{MR2177465,MR1976079, MR1976080} in our computations.
This method provides powerful functional analytic tools for 
studying invariant manifolds.  The core of the parameterization 
method is an invariance equation which conjugates a chart map  
for the local stable/unstable manifold to the linear dynamics given by the eigenvalues    
(see for example \eqref{eq:invariance_equation} 
in Section~\ref{sec:para_method}). Expanding the invariance equation as a formal 
series and matching like powers leads to 
\text{homological equations} for the 
coefficients of the series.  These homological 
equations are solved to any desired order, yielding a 
finite approximation.

Given a finite approximate parameterization we would like to evaluate
the associated truncation error. 
An important feature of the parameterization method is that there is a
natural notion of a posteriori error, i.e. one can  ``plug'' the approximate solution 
back into the invariance equation
and measure the distance from zero in an appropriate norm.
Further analysis is of course necessary in order to obtain validated error bounds,
as small defects need not imply small truncation errors.
When the invariance equation is formulated on a regular enough function space it is possible to apply a Newton-Kantorovich argument to get the desired bounds.

A uniqueness result for the parameterization method states that  
the power series coefficients are unique up to the choice of the scalings of the 
(un)stable eigenvectors \cite{MR1976079}.  This freedom in the choice of scaling can be 
exploited in order to control the numerical properties of the scheme.
For example by increasing or decreasing the length of the eigenvectors it is possible to 
manipulate the decay rates of the power series coefficients, and thus influence the
numerical stability of the scheme.  

One of the main findings of the present work is that the bounds required in the Newton-Kantorovich argument (see the definition of the \textit{radii polynomials} bounds in \eqref{eq:rad_poly_condition}) depend in an explicit way on the choice of the eigenvector scalings.  This result leads to algorithms for 
optimizing the choice of eigenvectors scalings under some fixed constraints. 
The algorithms developed in the present work complement similar automatic schemes developed in  \cite{HLM} (for computer assisted 
study of periodic orbits) and are especially valuable in continuation arguments, where one wants to compute the invariant manifolds
over a large range of parameter values in an automatic way.

\begin{remark}\emph{
The optimization constraints referred to above can be chosen in 
different ways depending on ones goals.  For example when the 
goal of the computation is visualization of the manifold it is 
desirable to choose scalings which maximize the ``extent'' of the manifold
in phase space (i.e. maximize the surface measure of the patch). 
On the other hand when the eigenvalues
have different magnitudes then it may be desirable to maximize the image
of the manifold under the constraint that the 
ratios of the scalings of the eigenvectors are fixed (this is especially 
useful in ``fast-slow'' systems).  In other situations one might want to 
optimize some other quantity all together.
Whatever constraints one chooses, we always want to optimize
while holding the error of the computation below some specified tolerance.
The main point of the present work is that whatever the desired constraints,
the explicit dependency of the bounds on the scaling facilitates the design of
algorithms which respect the specified error tolerance.  }
\end{remark}
\begin{remark}\emph{
We fix the domain of our approximate
parameterization to be the unit ball in $\mathbb{C}^m$ 
(where $m$ is the number of (un)stable eigenvalues, i.e. the 
dimension of the manifold) and vary the scalings of the eigenvectors
in order to optimize with respect to the constraints.  
Another (theoretically equivalent approach) would be to fix the scalings
of the eigenvectors and vary the size of the domain.  However the 
scalings of the eigenvectors determine the decay rates of the 
power series coefficients,
and working with analytic functions of fast decay seems to
stabilize the problem numerically.}  
\end{remark}
\begin{remark}\emph{In
many previous applications of the 
parameterization method the free constants were selected by some 
``numerical experimentation.'' See for example the introduction and 
discussion in Section $5$ of \cite{BDLM}, Remark $3.6$ of 
\cite{MR2821596}, Remark $2.18$ and $2.20$ of 
\cite{doi:10.1137/140960207}, the discussion of Example $5.2$ in
 \cite{MR2728178}, Remark $2.4$ of \cite{MR3068557},
and the discussion in Sections $4.2$ and $6$ of 
\cite{MR3068557}.
This motivates the need for systematic procedures 
developed here.
}
\end{remark}
\begin{remark}\emph{
The algorithms developed here facilitate the computation of 
local stable/unstable manifolds.  Once the local computations have
been optimized one could extend or ``grow'' larger 
patches of the local manifold using adaptive integration/continuation 
techniques.  This is a topic of substantial research and we refer the interested 
reader to the survey article \cite{MR2136745}.  
 See also the works of 
\cite{MR1713086, MR1870261, MR1981055, MR2835474, MR3021639, wittig_thesis, MR2644324}
and the references therein.
Combining these integration/continuation algorithms with the methods of the 
present work could be an interesting topic for future research.
}
\end{remark}
\begin{remark}\emph{
In the present work we employ a functional analytic 
style of a-posteriori analysis in conjunction with 
the parameterization method of  
\cite{MR1976079, MR1976080, MR2177465}.
Moreover the arguments are framed in classical weighted sequences spaces
following the work of \cite{MR648529, MR727816}.
There are in the literature many other methods for obtaining
rigorous computer assisted error bounds on numerical approximations of 
invariant manifolds.  
The interested reader should consult the works of 
\cite{MR2821596,  MR3068557, wittig_thesis, MR2644324, MR1236201, 
MR2947932, MR2784613, MR2461822, MR1961956, MR2173545,
MR2902618, MR3097025, MR2262261,
 MR3022075, MR3281845, MR2271217}
for other approaches and results.  
}
\end{remark}
\begin{remark}\emph{
In recent years a number of authors have developed numerical methods 
based on the parameterization method in order to compute invariant manifolds
of fixed and equilibrium points (e.g. see 
\cite{MR2728178, MR3079670, JDMJ01} for more discussion).  
The parameterization method can also be used  
to compute stable/unstable manifolds associated with 
periodic orbits of differential equations \cite{doi:10.1137/140960207, MR2551254, MR3118249},
as well as stable/unstable manifolds associated with invariant 
circles/tori \cite{MR2240743, MR2851901}.  Indeed the parameterization method can be 
extended in order to compute the invariant tori themselves  
\cite{MR2967458}, leading to a KAM theory ``without action angle
coordinates''.  
For more complete discussion of numerical methods 
based on the parameterization method 
we refer to the upcoming book \cite{mamotreto}.
For the moment we remark that the optimization
algorithms developed in the present 
work could be adapted to these more general settings.
}
\end{remark}

Our paper is organized as follows. In Section~\ref{sec:para_method} we present briefly the parameterization method and discuss its behaviour with respect to some specific changes of variable. In Section~\ref{sec:maximize} we give a way to numerically compute an approximate parameterization and then address the issue of finding a rescaling that maximize the image of the parameterization, while verifying some \textit{a posteriori} bounds that ensure (in some sense) the validity of the approximate parameterization. One possible way of proving the validity of the approximation is to use the ideas of rigorous computation, which we detail in Section~\ref{sec:rigorous}. We conclude in Section~\ref{sec:examples} by presenting the results obtained with our method to compute maximal patches of local manifolds for several examples.
The codes for all the examples can be found at \cite{webpage}.

\section{The parameterization method}  \label{sec:para_method}

In this section, we introduce the parameterization method for the stable manifold of an equilibrium solution of a vector field. 
The unstable manifold is obtained by time reversal.

\subsection{Invariance equation for stable manifolds of equilibria of vector fields}

We consider an ordinary differential equation (ODE) of the form
\begin{equation} \label{eq:vectorField}
y'=g(y), 
\end{equation}
where $g:\mathbb{R}^n \to \mathbb{R}^n$ is analytic. Assume that $p \in \R^n$ is an equilibrium point, i.e. $g(p)=0$, and assume that the dimension of the stable manifold at $p$ is given by $n_s \le n$. Denote $(\lambda_k,V_k)$, $1\leq k \leq n_s$ the stable eigenvalues (that is $\Re(\lambda_k)<0$, for $k=1,\dots, n_s$) together with associated eigenvectors, and denote $\Lambda=diag(\lambda_1,\ldots,\lambda_{n_s})$. 

We want to find an analytic parameterization of the local stable manifold at $p$. % (which is guarantee to exist if $g$ is analytic for instance). 
So we look for a power series representation
\begin{equation} \label{eq:f_power_series}
f(\theta)=\sum_{\vert\alpha\vert\geq 0}a_{\alpha}\theta^{\alpha},\quad 
\theta=\begin{pmatrix}\theta_1\\ \vdots\\ \theta_{n_s}\end{pmatrix} \in \mathbb{R}^{n_s},
\ a_{\alpha} = \begin{pmatrix}a_{\alpha}^{(1)} \\ \vdots \\ a_{\alpha}^{(n)}\end{pmatrix} \in\mathbb{R}^n,
\end{equation}
with the classical multi-indexes notations $\vert\alpha\vert=\alpha_1+\ldots+\alpha_{n_s}$ and  $\theta^{\alpha}=\theta_1^{\alpha_1}\ldots\theta_{n_s}^{\alpha_{n_s}}$, and assume that the parameterization $f$ conjugates the flow $\varphi$ induced by $g$ to the linear flow induced by $\Lambda$, that is
\begin{equation*}
f\left(e^{\Lambda t}\theta\right)=\varphi(t,f(\theta)).
\end{equation*}
Differentiating with respect to $t$ and taking $t=0$, we get that $f$ satisfies the invariance equation
\begin{equation}
\label{eq:invariance_equation}
Df(\theta)\Lambda \theta = g(f(\theta)),
\end{equation}
and to get a well-posed problem we add the following constraints
\begin{equation}
\label{eq:additional_conditions}
f(0)=p,\quad Df(0)=\begin{pmatrix} V_1 & \ldots & V_{n_s} \end{pmatrix}.
\end{equation}

Endow $\mathbb C^{n_s}$ with norm $\| \theta \|_{\mathbb C^{n_s}} = \max \{ |\theta_k| :k=1,\dots,n_s \}$, where $| \cdot |$ denotes the complex modulus, and using that norm,
denote by $B_{\nu} \subset \mathbb C^{n_s}$ the closed ball of radius $\nu$ centered at $0$.
We look for a parameterization $f$ which is analytic on a
ball $B_\nu \subset \mathbb{C}^{n_s}$ with $\nu > 0$.
We call the image $f[B_\nu]$ a \textit{patch} of the local invariant manifold.  

\begin{remark} \label{rem:complex_to_real}
{\em
If some of the eigenvalues happen to be complex-conjugate, say $\overline {\lambda_1}=\lambda_2,\ldots,\overline{\lambda_{2m-1}}=\lambda_{2m}$, it is easier to consider a power series $f$ with complex coefficients (i.e. with $a_{\alpha}\in\C^n$) and acting on $\theta\in \mathbb{C}^{n_s}$. We can then recover the real parameterization by considering, for $\theta\in\mathbb{R}^{n_s}$,
\begin{equation*}
f_{real}(\theta_1,\ldots,\theta_{n_s})=f(\theta_1+i\theta_2,\theta_1-i\theta_2,\ldots,\theta_{2m-1}+i\theta_{2m},\theta_{2m-1}-i\theta_{2m},\theta_{2m+1},\ldots,\theta_{n_s}).
\end{equation*}
See \cite{MR3207723} for a more detailed explanation of this fact. To be general in the sequel of our presentation, we will assume that $f$ is a complex power series.
}
\end{remark}

\begin{remark}\emph{
We say that there is a resonance of order 
$\alpha$ between the stable eigenvalues if 
\begin{equation} \label{eq:resDef}
\alpha_1 \lambda_1 + \ldots + \alpha_{n_s} \lambda_{n_s} = \lambda_j
\end{equation}
for some $1 \leq j \leq n_s$.  If there is no resonance for any $\alpha \in \mathbb{N}^{n_s}$
then we say that the stable eigenvalues are non-resonant.  Note that if $|\alpha|$ is 
large enough then a resonance is impossible.   
}

\emph{
It is shown in \cite{MR1976079} that if $g$ is analytic then
\eqref{eq:invariance_equation} has an analytic solution $f$ 
as long as the eigenvalues are non-resonant.  Moreover the power series coefficients 
of $f$ are uniquely determined up to the choice of the scalings of the eigenvectors.
This abstract result does not however provide explicit bounds on the size of the domain 
of analyticity $B_\nu$ for the parameterization: hence the need for a-posteriori validation of 
our numerical computations.
We also note that if there is a resonance then the invariance equation can be 
modified so that we conjugate to a polynomial (instead of linear) dynamical 
system \cite{MR1976079, parmChristian}, and that the later work just cited
implements computer assisted error bounds for the resonant case
using the radii polynomial approach.
Adapting the methods of the 
present work to the resonant case will be the topic of a future study.  It is 
clear from the work of \cite{MR1976079} that in the resonant case
the Taylor coefficients of the 
parameterization are unique up to the choice of the eigenvector scalings.
What remains to be checked is that in the resonant case
the eigenvector scalings appear in the radii polynomials 
in an explicit way (as is the case in for non-resonant eigenvalues, see Section 
\ref{sec:rigorous}).
}
\end{remark}

\subsection{Change of coordinates}
\label{sec:change_of_coordinates}

Assume that $f$ is a power series of the form \eqref{eq:f_power_series} satisfying \eqref{eq:invariance_equation} and \eqref{eq:additional_conditions} (therefore it is a local parameterization of the stable manifold at $p$). Now consider a change of coordinates in $\mathbb{C}^{n_s}$, defined by some invertible matrix $\Gamma\in M_{n_s}(\C)$, and the new power series 
\begin{equation*}
\tilde f(\theta)=f(\Gamma\theta).
\end{equation*}

Thanks to (\ref{eq:invariance_equation}), we have that
\begin{equation} \label{eq:new_invariance_equation}
D\tilde f(\theta)\Lambda \theta = Df(\Gamma\theta)\Gamma\Lambda \theta = g(f(\Gamma\theta)) = g(\tilde f(\theta)).
\end{equation}
%\begin{equation*}
%g(\tilde f(\theta)) = g(f(\Gamma\theta)) = Df(\Gamma\theta)\Lambda \Gamma\theta.
%\end{equation*}
So if $\Gamma$ is such that $\Gamma\Lambda=\Lambda \Gamma$, then $\tilde f$ also satisfies the invariance equation (\ref{eq:invariance_equation}), together with the slightly modified conditions
\begin{equation} \label{eq:rescaled_parameters}
\tilde f(0)=p,\quad D\tilde f(0)=\Gamma\begin{pmatrix} V_1 & \ldots & V_{n_s} \end{pmatrix}.
\end{equation}
\begin{remark}
{\em
From now on we assume that $\Gamma=\diag(\gamma_1,\ldots,\gamma_{n_s})$, which is sufficient to have $\Gamma\Lambda=\Lambda \Gamma$ (it is also necessary if the $\lambda_k$ are pairwise distinct). We also assume that the $\gamma_i$ are all real positive numbers and that coefficients $\gamma_i$ corresponding to two complex conjugates eigenvalues are equal. Taking $\gamma_i$ real is natural if all the eigenvalues are real (and $f$ is therefore a real power series). On the other hand if there are some complex-conjugate eigenvalues, say $\lambda_1=\overline{\lambda_2}$, then the recovery of a real parameterization as explained in Remark~\ref{rem:complex_to_real} uses the fact that the corresponding eigenvectors $V_1$ and $V_2$ also are complex-conjugate, and that this property is propagated to all the coefficients of the parameterization when recursively solving the invariance equation \eqref{eq:expanded_invariance_equation}. By taking $\gamma_1$ and $\gamma_2$ real and equal, we ensure that this property is conserved after the rescaling (namely $\gamma_1 V_1=\overline{\gamma_2 V_2}$), so that we can still easily recover a real parameterization. Admittedly, we could relax this hypothesis and only assume that $\gamma_1$ and $\gamma_2$ themselves are complex-conjugate, but we will not consider this possibility here.
}
\end{remark}

As announced, we now consider $\Gamma=\diag(\gamma_1,\ldots,\gamma_{n_s})$, where $\gamma_i\in\R_+$ for all $1\leq i\leq n_s$, and $\tilde f$ defined as $\tilde f(\theta)=f(\Gamma\theta)$. The above discussion shows that $\tilde f$ is a new parameterization of the local manifold, since it satisfies \eqref{eq:new_invariance_equation} and \eqref{eq:rescaled_parameters}.
%
%
%
%\begin{equation*}
%D\tilde f(\theta)\Lambda \theta = g(\tilde f(\theta)),
%\end{equation*}
%together with
%\begin{equation*}
%\tilde f(0)=p,\quad D\tilde f(0)=\Gamma\begin{pmatrix} V_1 & \ldots & V_{n_s} \end{pmatrix}.
%\end{equation*}
Besides, the Taylor expansion of $\tilde f$ can be easily expressed in terms of the Taylor expansion of $f$. Indeed if we write $\tilde f$ as
\begin{equation*}
\tilde f(\theta)=\sum_{\vert\alpha\vert\geq 0}\tilde a_{\alpha}\theta^{\alpha},
\end{equation*}
then the coefficients are given by
\begin{equation} \label{eq:rescaling}
\tilde a_{\alpha} = a_{\alpha}\gamma^{\alpha},
\end{equation}
where $\gamma=(\gamma_1,\ldots,\gamma_{n_s})$ and again standard multi-indexes notations. Therefore it is enough to find one parameterization $f$ of the local manifold (or more precisely  its coefficients $a_{\alpha}$) to get all the re-parameterizations $\tilde f$ (at least those given by a diagonal matrix $\Gamma$) without further work. Let us introduce an operator acting on sequences to express this rescaling in a condensed way.
\begin{definition}
Given $\gamma=(\gamma_1,\ldots,\gamma_{n_s})$, we define $\cL$ (acting on $a$) component-wise by
\begin{equation*}
\cL_{\alpha}(a) = \gamma^{\alpha}a_{\alpha},\quad \forall \vert\alpha\vert\geq 0.
\end{equation*} 
\end{definition}
Therefore, if $a$ is the sequence of coefficients of the parameterization $f$, the sequence of coefficients of the parameterization $\tilde f$ defined as above is given by $\cL(a)$.

\section{How to compute \boldmath $f$ \unboldmath and maximize the local manifold patch}
\label{sec:maximize}

In this section we present a method to compute numerically a parameterization of the manifold (that is the coefficients $(a_{\alpha})$) and then choose a proper rescaling $\gamma$ to maximize the corresponding image. We assume in the sequel that the nonlinearities in $g$ are polynomials. Note that this not so restrictive as it might first seems, as techniques of automatic differentiation 
can be used in order to efficiently compute the  (Taylor/Fourier/Chebyshev) series expansions of compositions 
with elementary functions.  The authors first learned these techniques from Chapter 4.7 of \cite{MR3077153},
but the interested reader should also refer to  
the discussion and references in \cite{alexADmanuscript, MR2146523}.

Automatic differentiation is also a valuable tool for validated numerics, as polynomial nonlinearities are 
often more convenient to work with than transcendental ones.  
Since elementary functions of mathematical physics (powers, exponential, trigonometric functions, rational, Bessel, elliptic integrals, etc.) are themselves solutions of ODEs, these ODEs can be appended to the original problem of interest in order to obtain a new problem with only polynomial nonlinearities (but with more variables and more equations).  
Moreover, in many computer assisted proofs it is the dimension of the underlying invariant object, and not
the dimension of the embedding space, that informs the difficulty of the problem.
We refer for example the book of \cite{MR2807595} for a much more complete discussion of these
matters. We also mention that automatic differentiation has been combined with the radii 
polynomial approach in \cite{automatic_diff_fourier} in order to compute periodic orbits of some 
celestial mechanics applications.  

Of course automatic differentiation is not the only method which can be used in 
order to replace a transcendental vector field with a polynomial one.   Any method 
of polynomials approximation can be used.  A detailed survey of 
the interpolation literature is far beyond the scope of the present work, 
however we mention the works of \cite{MR2259202, MR3124898} where one can 
find implementation details and fuller discussion of the literature surrounding the  use 
of Chebyshev polynomials to expand transcendental nonlinearities and obtain computer 
assisted error bounds.  We also note that general purpose software exists 
for carrying out these kinds of manipulations, even with 
mathematically rigorous error bounds \cite{MR1652147, MR2534406, MR1420838}.

\subsection{Computation of the approximate parameterization} \label{sec:approximation}

Let $f$ be a power series as in \eqref{eq:f_power_series}, assume $g$ is a polynomial vector field of degree $d$ given by
\begin{equation*}
  g(y)=\sum_{\vert\beta\vert\leq d} b_{\beta}y^{\beta}, \qquad b_{\beta} \in \R^n
\end{equation*}
and plug it into the invariance equation \eqref{eq:invariance_equation}. We obtain 
\begin{equation} \label{eq:expanded_invariance_equation}
\sum_{\vert\alpha\vert\geq 0} \left(\alpha_1\lambda_1 + \ldots + \alpha_{n_s}\lambda_{n_s}\right)a_{\alpha}\theta^{\alpha} = \sum_{\vert\beta\vert \leq d} b_{\beta}\left(\sum_{\vert\alpha\vert\geq 0} a_{\alpha}\theta^{\alpha}\right)^{\beta} = \sum_{\vert\alpha\vert\geq 0} \sum_{\vert\beta\vert \leq d} b_{\beta}\left( a^{\beta}\right)_{\alpha}\theta^{\alpha},
\end{equation}
where again we use multi-indexes notations, $a^{\beta}=\left(a^{(1)}\right)^{\beta_1} \ast \ldots \ast \left(a^{(n)}\right)^{\beta_n}$, and $\ast$ denotes the Cauchy product. Notice that the two conditions in \eqref{eq:additional_conditions} imply that the coefficients of order $0$ and $1$ are the same on both sides of \eqref{eq:expanded_invariance_equation}. There are several ways to obtain an approximation of the coefficients $\left(a_{\alpha}\right)_{\vert\alpha\vert\geq 2}$ so that \eqref{eq:expanded_invariance_equation} is satisfied, one of them being to compute them recursively for increasing $\vert\alpha\vert$. Here we present another method, which fits naturally with the ideas of rigorous computations exposed later in the paper. We define the infinite dimensional vector $a=\left(a_{\alpha}\right)_{\vert\alpha\vert\geq 0}$ and the operator $F$, acting on $a$ component-wise by
\begin{equation*}
F_{\alpha}(a) =
\left\{
\begin{aligned}
& a_0 - p, \quad &\text{if }\alpha=0, \\
& a_{e_i} - V_i,\quad &\text{if }\alpha = e_i,\ \forall~1\leq i \leq n_s \\
& \left(\alpha_1\lambda_1 + \ldots + \alpha_{n_s}\lambda_{n_s}\right)a_{\alpha} - \sum_{\vert\beta\vert \leq d} b_{\beta}\left( a^{\beta}\right)_{\alpha}, \quad&\forall~\vert\alpha\vert\geq 2.
\end{aligned}
\right.
\end{equation*}
Finding $a$ solving \eqref{eq:expanded_invariance_equation} and the additional 
conditions \eqref{eq:additional_conditions} is equivalent to solve
\begin{equation} \label{eq:F=0}
F(a) = \left\{ F_{\alpha}(a) \right\} _{|\alpha| \ge 0}  = 0.
\end{equation}
Given $\gamma=(\gamma_1,\ldots,\gamma_{n_s})$, finding a rescaled parameterization (that is solving \eqref{eq:invariance_equation} and \eqref{eq:rescaled_parameters}) can also be expressed as finding the zero of the function $\tilde F$, which is defined the same way as $F$ except for the indices $\vert\alpha\vert=1$:
\begin{equation}
\label{eq:F_tilde}
\tilde F_{\alpha}(a) =
\left\{
\begin{aligned}
& a_0 - p, \quad &\text{if }\alpha=0, \\
& a_{e_i} - \gamma_iV_i,\quad &\text{if }\alpha = e_i,\ \forall~1\leq i \leq n_s \\
& \left(\alpha_1\lambda_1 + \ldots + \alpha_{n_s}\lambda_{n_s}\right)a_{\alpha} - \sum_{\vert\beta\vert \leq d} b_{\beta}\left( a^{\beta}\right)_{\alpha}, \quad&\forall~\vert\alpha\vert\geq 2.
\end{aligned}
\right.
\end{equation}
Notice that the discussion in Section~\ref{sec:change_of_coordinates} shows that $F(a)=0$ if and only if $\tilde F(\cL(a))=0$.
\begin{remark}
{\em
Since $a_0$ and the $a_{e_i}$ are fixed by the additional conditions~\eqref{eq:additional_conditions}, we could also consider them as parameters and define $F$ as $(F_{\alpha})_{\vert\alpha\vert\geq 2}$, acting only on $(a_{\alpha})_{\vert\alpha\vert\geq 2}$. We do this for the examples of Sections~\ref{sec:Lorenz} and \ref{sec:FHN}, but we keep the above definition of $F$ and $\tilde F$ when we use rigorous computation (Section~\ref{sec:rigorous} and example in Section~\ref{sec:bridge}), because it allows for a simpler presentation.
}
\end{remark}

Now we fix an integer $N$ and define the truncated operator $F^{[N]}=\left(F_{\alpha}\right)_{\vert\alpha\vert<N}$, acting on a truncated sequence $a^{[N]}=\left(a_{\alpha}\right)_{\vert\alpha\vert<N}$, by
\begin{equation*}
F_{\alpha}(a^{[N]}) =
\left\{
\begin{aligned}
& a_0 - p, \quad &\text{if }\alpha=0, \\
& a_{e_i} - V_i,\quad &\text{if }\alpha = e_i,\ \forall~1\leq i \leq n_s \\
& \left(\alpha_1\lambda_1 + \ldots + \alpha_{n_s}\lambda_{n_s}\right)a_{\alpha} - \sum_{\vert\beta\vert \leq d} b_{\beta}\left( a^{\beta}\right)_{\alpha}, \quad&\forall~2\leq \vert\alpha\vert<N.
\end{aligned}
\right.
\end{equation*}
Since the problem is now finite dimensional, we can use Newton's method to compute an approximate zero of $F^{[N]}$. In the rest of this paper, $\ba$ will denote such an approximate solution completed with 0 for $\vert\alpha\vert\geq N$. See Section~\ref{sec:examples} for explicit examples. Also note that the only property that really matters concerning the approximate parameterization $\ba$ is that $\ba_{\alpha}=0$ for all $\vert\alpha\vert\geq N$. As long as it satisfies this property, everything in the sequel will work, even if $\ba$ was obtained in a different fashion than the one we just presented (for instance by solving inductively a finite number of homological equations).

\begin{remark}
{\em
Taking $N$ larger leads to a better approximation but at the expense of computational cost, so its choice depends on how precise an approximation you need, and how much computational resources you have.
}
\end{remark}

\subsection{Maximizing the image of the parameterization}

Now that we have an approximate parameterization, we focus on maximizing the image of the corresponding manifold, while checking that our approximation is still valid. 
The power series $f$ given by \eqref{eq:f_power_series} is now considered as 
\[
f:B_\nu \to \C^n
\]
for some $\nu > 0$. One approach in getting the largest possible image of $f$ would be to maximize the $\nu$ for which \eqref{eq:invariance_equation} is {\em valid} on $B_{\nu}$. We give in Definition~\ref{def:defect_valid} and Definition~\ref{def:proof_valid} two different definitions of parameterization validity. 

\begin{remark}
{\em
For reasons of numerical stability, we always consider the parameter space $B_{\nu}$ for $\nu=1$ and instead use the $\gamma$ introduced in the reparameterization of Section~\ref{sec:change_of_coordinates} as a parameter. Indeed, assume that the parameterization $f$ is valid on $B_{\nu_1}$ for some $\nu_1$, then proving that it still is on $B_{\nu_2}$ for a different $\nu_2$ is equivalent to prove that $\tilde f(\theta)=f(\Gamma \theta)=f(\gamma_1\theta_1,\ldots,\gamma_{n_s}\theta_{n_s})$ is valid on $B_{\nu_1}$, with $\gamma_k=\frac{\nu_2}{\nu_1}$ for all $k$. So we can always keep $\nu=1$ and rather try to maximize the $\gamma_k$ for which $\tilde f$ is valid on $B_{1}$.
}
\end{remark}

Based on the previous remark, from now on, and for the rest of the present paper, we always fix $\nu=1$, and therefore drop all references to this parameter.

\begin{remark}
{\em
If the eigenvalues are real and not all equal to the same value, it may be useful to consider different scalings for each direction, that is to take $\Gamma=\diag(\gamma_1,\ldots,\gamma_{n_s})$ with different $\gamma_k$ rather than $\Gamma=\diag(\gamma,\ldots,\gamma)$. Indeed in this work we aim at maximizing the surface of the manifold patch, but for some specific problem (a fast-slow system for instance), you may rather want to enlarge the manifold in one precise direction, in which case you should definitely consider different $\gamma_k$ for each $k$.
}
\end{remark}

In this paper we will use two different criteria to say that our parameterization is valid on $B_{\nu}$. The first one is a numerical a posteriori estimate and the second is a rigorous validation.
In order to measure the validity of a parameterization, we need to compute the norm of a sequence $a = \{ a_{\alpha} \}_{|\alpha| \ge 0}$ with $a_{\alpha} \in \C^n$. For this, let us introduce the space
\[
\ell_\nu^1 \bydef \left\{ u = \{ u_{\alpha} \}_{|\alpha| \ge 0} \mid u_{\alpha} \in \C \text{ and } \| u \|_{\ell_\nu^1} \bydef \sum_{|\alpha| \ge 0} |u_{\alpha}| \nu^{|\alpha|}  < \infty \right\}.
\]
Given $a=\left(a_{\alpha}\right)_{\vert\alpha\vert\geq 0}$, with $a_{\alpha} = \begin{pmatrix}a_{\alpha}^{(1)} \\ \vdots \\ a_{\alpha}^{(n)}\end{pmatrix} \in\C^n$, 
denote $a^{(i)}=\left(a^{(i)}_{\alpha}\right)_{\vert\alpha\vert \geq 0}$. Then, consider the product space
\[
X \bydef \left(\ell_\nu^1 \right)^n 
\bydef 
\left\{
a = \left(a_{\alpha}\right)_{\vert\alpha\vert\geq 0} \mid 
\left\Vert a\right\Vert_{X} \bydef  \max\limits_{1\leq i\leq n} \left\Vert a^{(i)}\right\Vert_{\ell_\nu^{1}} <\infty
\right\}.
\]
\begin{remark}
\label{rem:ordering}\emph{
It will be usefull to represent linear operators acting on elements of $X$ with (infinite) matrix/vector notations. To prevent any future ambiguity, let us precise the ordering we use in this paper for those vectors and matrices. Given $a\in X$, we represent it as the (infinite) vector $\left(a_{\alpha}\right)_{\vert\alpha\vert\geq 0}$ where the $a_{\alpha}$ are ordered by growing $\vert\alpha\vert$, and by lexicographical order within the coefficients with same $\vert\alpha\vert$. For instance, if $n_s=2$, 
\begin{equation*}
a=\!\left(a_{0,0},a_{1,0},a_{0,1},a_{2,0},a_{1,1},a_{0,2},\ldots\right)^T.
\end{equation*}
Notice that each $a_{\alpha}$ is himself a vector of $\R^n$. For an (infinite) matrix $M=\left(M_{\alpha,\beta}\right)_{\vert\alpha\vert,\vert\beta\vert\geq 0}$ representing a linear operator on $X$, we use the same order for the rows and columns. Notice that each coefficient $M_{\alpha,\beta}$ is in fact a $n$ by $n$ matrix whose coefficient on row $i$ and column $j$ will be denoted as $M^{(i,j)}_{\alpha,\beta}$, so that
\begin{equation*}
\left(Ma\right)^{(i)}_{\alpha}=\sum_{\vert\beta\vert\geq 0}\sum_{j=1}^n M^{(i,j)}_{\alpha,\beta}a^{(j)}_{\beta}.
\end{equation*}
}\end{remark} 

We now give the two announced criteria to measure the validity of a parameterization.
\begin{definition} \label{def:defect_valid}
Fix a {\em defect} threshold $\varepsilon_{max}>0$, a truncation dimension $N$ and an approximate solution $\ba^{[N]}$ computed using the method of Section~\ref{sec:approximation}. Denote $\ba=\ba^{[N]}$. We say that 
\begin{equation} \label{eq:finite_dim_parameterization}
f(\theta) \bydef \sum_{|\alpha|<N} \ba_{\alpha} \theta^{\alpha}
\end{equation}
is {\em defect-valid} on $B_{\nu}$ if 
\begin{equation} \label{eq:defect_valid}
\Vert F(\ba)\Vert_{X} < \varepsilon_{max}.
\end{equation}
Equivalently, we say that $\ba$ is {\em defect-valid} on $B_{\nu}$ if \eqref{eq:defect_valid} holds. Given $\gamma=(\gamma_1,\ldots,\gamma_{n_s})$, we also say that the rescaled parameterization $\cL(\ba)$ is {\em defect-valid} on $B_{\nu}$ if
\begin{equation} \label{eq:defect_valid_rescaled}
\Vert \tilde F(\cL(\ba))\Vert_{X} < \varepsilon_{max}.
\end{equation}
\end{definition}

Remember that $g$ is assumed to be polynomial, and so $F$ is also polynomial, say of degree $d$. Since $\ba_{\alpha}=0$ for $\vert \alpha\vert \geq N$, then $F_{\alpha}(\ba)=0$ for all $\vert\alpha\vert\geq d(N-1)+1$. Thus the quantity $\Vert F(\ba)\Vert_{X}$ in \eqref{eq:defect_valid} is only a finite sum and can be computed explicitly. 

Assume now that we have computed all the $F_{\alpha}(\ba)$ for $\vert\alpha\vert\leq d(N-1)$ (which can be quite long because of the Cauchy products coming from the nonlinearities). When we then consider some $\gamma=(\gamma_1,\ldots,\gamma_{n_s})$ and the rescaled parameterization $\cL(\ba)$, we get (using the fact that the nonlinearities are polynomial and the definition of the Cauchy product) that for all $\vert\alpha\vert \geq 0$,
\begin{equation}
\label{eq:F_rescaled}
\tilde F_{\alpha}(\cL(\ba))=\gamma^{\alpha}F_{\alpha}(\ba).
\end{equation}
This way, the evaluation of $\Vert \tilde F(\cL (\ba))\Vert_{X}$ for any rescaling is computationally cheap and thus it is rather straightforward to find the $\gamma$ for which the re-parameterization $\cL(\ba)$ gives the largest image of the manifold, while being defect-valid. Let us be a little more precise about this. Depending on our goal we use two different approaches.

\paragraph*{Method 1:} We look for eigenvector scalings which 
maximize the surface measure, subject to the restriction that 
the rescaled parameterization is defect-valid. Therefore we find numerically a mesh of the compact set 
\[
\left\{ \gamma \in \mathbb{R}_+^{n_s} \mid \Vert \tilde F(\cL(\ba))\Vert_{X}=\varepsilon_{max} \right\}
\]
and then approximately compute the surface area of the image for each point of the mesh. We refer to Sections~\ref{sec:Lorenz} and \ref{sec:FHN} for explicit examples in dimension 2.

\paragraph*{Method 2:} We want to emphasize some specific directions when computing the manifold. Therefore we fix some weights $\omega_1,\ldots,\omega_{n_s}$ and consider only rescalings of the form 
\begin{equation*}
\gamma=\gamma(t)=(t\omega_1,\ldots,t\omega_{n_s}).
\end{equation*}
We then look for the largest $t$ such that the rescaled parameterization is defect-valid. By doing so we obtain a manifold that stretches more in the directions with the largest weights. We refer to Sections~\ref{sec:Lorenz} and \ref{sec:FHN} for explicit examples in dimension 2 where we stretch the manifolds in the slow direction.

\begin{remark}\emph{
When there is only one stable/unstable eigenvalue (or a single pair of complex conjugate 
eigenvalues) then Method 2 reduces to choosing the 
largest possible scaling for the eigenvector (or for the complex conjugate pair of eigenvectors)
so that the rescaled parameterization is defect-valid.}
\end{remark}

Now we would like to present a different definition of validity of a parameterization, inspired by the field of rigorous computing. For this, we briefly review the ideas of rigorous computation.
The idea is to reformulate the problem $F(a)=0$ given in \eqref{eq:F=0} and to look for a fixed point of a Newton-like equation of the form 
\begin{equation*} 
T(a)=a-AF(a)
\end{equation*}
where $A$ is an approximate inverse of $DF(\ba)$, and $\ba$ is a numerical approximation obtained by computing a finite dimensional projection of $F$ (in our case we called it $F^{[N]}$). Let us explain how we construct $A$. Remembering that
\begin{equation*}
F_{\alpha}(a) = \left(\alpha_1\lambda_1 + \ldots + \alpha_{n_s}\lambda_{n_s}\right)a_{\alpha} - \sum_{\vert\beta\vert \leq d} b_{\beta}\left( a^{\beta}\right)_{\alpha}, \quad\forall~\vert\alpha\vert\geq 2
\end{equation*}
%
%and that we work in $\ell_\nu^1$ which is a Banach algebra of exponentially fast decaying coefficient, 
we consider the following approximation for $DF(\ba)$
\begin{equation*}
A^{\dag}=
\begin{pmatrix}
DF^{[N]}(\ba) & & 0 & \\
 & A_N^{\dag} & & \\
0 & & A_{N+1}^{\dag} & \\
 & & & \ddots\\
\end{pmatrix}
\end{equation*}
where for each $k\geq N$, $A_k^{\dag}$ is a finite bloc diagonal matrix, each of its diagonal block being of size $n$ and of the form $\left(\alpha_1\lambda_1 + \ldots + \alpha_{n_s}\lambda_{n_s}\right)I_n$, where $\vert\alpha\vert=k$ and $I_n$ is the $n$ by $n$ identity matrix. In other words (see Remark~\ref{rem:ordering})
\begin{equation*}
A_k^{\dag}\left(a_{\alpha}\right)_{\vert\alpha\vert=k} = \left(\left(\alpha_1\lambda_1 + \ldots + \alpha_{n_s}\lambda_{n_s}\right) a_{\alpha}\right)_{\vert\alpha\vert=k}.
\end{equation*}

We then define an approximate inverse $A$ of $DF(\ba)$ as 
\begin{equation}
\label{eq:A}
A \bydef
\begin{pmatrix}
A^{[N]} & & 0 & \\
 & A_N & & \\
0 & & A_{N+1} & \\
 & & & \ddots\\
\end{pmatrix},
\end{equation}
where $A^{[N]}$ is a numerical approximation of $DF^{[N]}(\ba)^{-1}$ while the $A_k \bydef \left(A_k^{\dag}\right)^{-1}$ are the exact inverses. We then prove the existence of a zero of $F$ by using a contraction argument yielding the existence of a fixed point of $T$. A precise theorem is stated below, but just before that we need (given $\gamma=(\gamma_1,\ldots,\gamma_{n_s})$) to define a \textit{rescaled} operator
\begin{equation}
\label{eq:T_tilde}
\tilde T \bydef I - \tilde A \tilde F 
\end{equation} 
that we can use in a similar fashion to prove the existence of a zero of $\tilde F$. Remembering that
\begin{equation*}
\tilde F (\cL(a))=\cL F(a)
\end{equation*}
we have 
\begin{equation*}
D\tilde F(\cL(a)) = \cL DF(a) \cL^{-1}
\end{equation*}
and therefore we consider
\begin{equation}
\label{eq:to_make_JP_happy}
\tilde A^{\dag} \bydef \cL A^{\dag} \cL^{-1} \quad \text{and} \quad \tilde A \bydef \cL A \cL^{-1}
\end{equation}
as approximations for $D\tilde F(\cL(\ba))$ and $\left(D\tilde F(\cL(\ba))\right)^{-1}$ respectively. 

The rigorous enclosure of a solution follows by verifying the hypothesis of the following Newton-Kantorovich type argument. Our method, often called the {\em radii polynomial approach}, was originally developed to study equilibria of PDEs 
\cite{MR2338393} and was strongly influenced by the work of Yamamoto \cite{MR1639986}. 
The differences between the radii polynomial approach and the standard Newton-Kantorovich approach are mainly twofold. First, the map $\tilde F$ under study is not required to map the Banach space $X$ into itself. This is often the case when the map $\tilde F$ comes from a differential equation and results in a loss of regularity of the function it maps. Second, the approach does not require controlling the exact inverse of the derivative, but rather only an approximate inverse. This can be advantageous as controlling  exact inverses of infinite dimensional linear operator can be challenging. For more details on the radii polynomial approach for rigorous computations of stable and unstable manifolds of equilibria, we refer to \cite{parmChristian}. Given $r>0$, denote by $B_{r}(a) \subset X = \left(\ell_\nu^1 \right)^n$ the ball centered at $a \in X$ of radius $r$. 

\begin{theorem} \label{th:fixed_point}
Let $\gamma=(\gamma_1,\ldots,\gamma_{n_s})\in\mathbb{R}_+^{n_s}$. Assume that the linear operator $A$ in \eqref{eq:A} is injective.
For each $i=1,\dots,n$, assume the existence of bounds $\tilde Y = \left( \tilde Y^{(1)} , \dots, \tilde Y^{(n)} \right)$ and $\tilde Z(r) = \left( \tilde Z^{(1)}(r) , \dots, \tilde Z^{(n)}(r) \right)$ such that
\begin{equation} \label{eq:Y_and_Z}
\left\Vert \left(\tilde T(\cL(\ba))-\cL(\ba)\right)^{(i)}\right\Vert_{\ell_\nu^1 }\leq \tilde Y^{(i)}\quad \text{and}\quad \sup\limits_{b,c\in B_{r}(0)} \left\Vert \left(D\tilde T(\cL(\ba)+b)c\right)^{(i)}\right\Vert_{\ell_\nu^1} \leq \tilde Z^{(i)}(r).
\end{equation}
If there exists $r>0$ such that
\begin{equation} \label{eq:rad_poly_condition}
\tilde P^{(i)}(r) \bydef \tilde Y^{(i)}+\tilde Z^{(i)}(r)-r<0, \quad \text{for all } i=1,\dots,n
\end{equation}
then $\tilde T:B_{r}(\cL(\ba)) \to B_{r}(\cL(\ba))$ is a contraction. By the contraction mapping theorem, there exists a unique $a^* \in B_{r}(\cL(\ba)) \subset X$ such that $\tilde F(a^*)=0$. 
Moreover, $\| a^* - \cL(\ba) \|_X \le r$.
\end{theorem}
As we see in Section~\ref{sec:rigorous}, the bounds $\tilde P^{(1)}(r), \dots, \tilde P^{(n)}(r)$ given in \eqref{eq:rad_poly_condition} can be constructed as polynomials in $r$ and are called the {\em radii polynomials}.

The statement of Theorem~\ref{th:fixed_point} is now used to define our second definition of validity of a parameterization, which is of course more costly than the first one but provides rigorous bounds. 

\begin{definition} \label{def:proof_valid}
Fix a {\em proof} threshold $r_{max}$, a truncation dimension $N$ and an approximate solution $\ba$. Given a numerical zero $\ba$ of $F$ and $\gamma=(\gamma_1,\ldots,\gamma_{n_s})$, we say that the parameterization $\cL(\ba)$ is {\em proof-valid} on $B_{\nu}$ if there exists $r>0$ such that condition \eqref{eq:rad_poly_condition} holds for some $r \le r_{max}$.
\end{definition}

In the next section we explain how the bounds $\tilde Y$ and $\tilde Z$ can be constructed so that they depend explicitly on the scaling $\gamma$. Then, as for Definition~\ref{def:defect_valid}, you only need to do the costly computations once for $\ba$ (that is for $\gamma=(1,\ldots,1)$) and then the new bounds (and thus the new radii polynomials $\tilde P^{(i)}$) can be computed easily for any rescaling. Therefore the process of finding the rescaling $\gamma$ which maximizes the image of a manifold given by a proof-valid parameterization is also rather straightforward. We give in Section~\ref{sec:bridge} an example of application where we explicitly compute the bounds $\tilde Y$ and $\tilde Z$.

\section{Explicit dependency of the radii polynomials in the scaling \boldmath $\gamma$ \unboldmath}
\label{sec:rigorous}

In this section we construct the bounds $\tilde Y$ and $\tilde Z$ satisfying \eqref{eq:Y_and_Z} with an explicit dependency on the $\gamma$ whose action is given by \eqref{eq:rescaling}.

\subsection{The bound \boldmath $\tilde Y$ \unboldmath}
\label{sec:Y}

\begin{proposition}
\label{prop:Y}
The bound $\tilde Y=(\tilde Y^{(1)},\dots,\tilde Y^{(n)})$ defined component-wise by 
\begin{equation*}
\tilde Y^{(i)}=\left\Vert \left(\cL A F(\ba)\right)^{(i)}\right\Vert_{\ell_\nu^1}, \quad \forall~1\leq i\leq n,
\end{equation*}
satisfies \eqref{eq:Y_and_Z}.
\end{proposition}

\begin{proof}
By definition of $\tilde T$, 
\begin{equation*}
\tilde T(\cL(\ba))-\cL(\ba)=\cL AF(\ba)
\end{equation*} 
and we have that
\begin{equation*}
\tilde A\tilde F(\cL(\ba))= \cL AF(\ba),
\end{equation*}
which yields the formula for $\tilde Y$.
\end{proof}

\begin{remark}
{\em
As previously mentioned, $F_{\alpha}(\ba)=0$ if $\vert\alpha\vert \ge d(N-1) + 1$, and since $A$ is of the form
\begin{equation*}
A=
 \left(\begin{array}{c|c}
 A^{[N]} & 
 \begin{matrix}
  & & & \\
  & & & \\
  & & &
 \end{matrix}\\
 \hline
  \begin{matrix}
  & & & \\
  & & & \\
  & & &
 \end{matrix}  & Tail 
 \end{array}\right),
\end{equation*}
where $Tail$ is a diagonal matrix (see~\eqref{eq:A}), then $\tilde Y$ can be computed as a finite sum. Moreover, the $\tilde Y$ bound can be expensive to evaluate, since it requires computing the Cauchy products involved in $F(\ba)$, the matrix $D$ which is the numerical inverse of the full and possibly large matrix $DF^{[N]}(\ba)$, and the product $A F(\ba)$. However, once $A F(\ba)$ is computed, we only need to do the component-wise multiplication defined by $\cL$ and the finite sum corresponding to the $\ell_\nu^1$ norm to get the bound $\tilde Y$ for any rescaling $\gamma$. Therefore, recomputing the bound $\tilde Y$ for a different rescaling is cheap.
}
\end{remark}

\subsection{The bound \boldmath $\tilde Z$ \unboldmath}
\label{sec:Z}

For the clarity of the exposition, we now assume that the nonlinearity of $g$ (and thus of $F$) are of degree 2. We insist that the method presented here still holds for nonlinearity of higher degree (see for instance \cite{BDLM,MR2821596}). It is also worth mentioning that in the context of computing equilibria of PDEs in \cite{MR2338393,MR3077902,MR2718657} and periodic orbits of delay differential equations in \cite{MR2871794}, the bounds of the radii polynomials have been derived for general polynomial problems. Here, we decided that staying fully general would only obscure the point with notations, hence our restriction to quadratic nonlinearities.

To compute the $\tilde Z$ bound, we split $D\tilde T(\cL(\ba)+b)c$ as
\begin{align*}
D\tilde T(\cL(\ba)+b)c &= \left(I-\tilde A\tilde A^{\dag}\right)c +\tilde A\left(D\tilde F(\cL(\ba)+b)\-\tilde A^{\dag}\right)c \\
&= \left(I-\tilde A\tilde A^{\dag}\right)c +\tilde A\left(D\tilde F(\cL(\ba))-\tilde A^{\dag}\right)c +D^2\tilde F(\cL(\ba))(b,c)
\end{align*}
and we are going to bound each term separately.

\subsubsection{The bound \boldmath $\tilde Z_0$ \unboldmath}
\label{sec:Z_0}

We start this section with a result providing an explicit formula for the $\ell_\nu^1$ operator norm of a matrix.
\begin{lemma}
\label{lem:B^{[N]}}
Let $\varrho_{n,n_s,N}=n\binom{N+n_s-1}{n_s}$ and $B\in M_{\varrho_{n,n_s,N}}(\C)$. For all $c\in \left(\ell_\nu^1\right)^n$,
\begin{equation*}
\left\Vert \left(Bc^{[N]}\right)^{(i)} \right\Vert_{\ell^{1}_{\nu}} \leq \sum_{j=1}^n K_{B}^{(i,j)} \left\Vert c^{(j)}\right\Vert_{\ell_\nu^1},
\end{equation*}
where
\begin{equation}
\label{eq:op_norm}
K_{B}^{(i,j)} = \max\limits_{0\leq \vert\beta\vert<N} \left(\frac{1}{\nu^{\vert\beta\vert}} \sum_{0\leq \vert\alpha\vert<N}  \left\vert B^{(i,j)}_{\alpha,\beta}\right\vert \nu^{\vert\alpha\vert}\right), \quad \forall~1\leq i,j\leq n.
\end{equation}
\end{lemma}
\begin{remark}
{\em The matrix/vector product should be understood according to Remark~\ref{rem:ordering} with $\varrho_{n,n_s,N}$ simply being the length of $\left(c_{\alpha}\right)_{\vert\alpha\vert<N}$ seen as a vector of complex numbers. Lemma~\ref{lem:B^{[N]}} is just the computation of the matrix norm associated to the weighted vector norm defined on $\ell^{1}_{\nu}$.}
\end{remark}
\begin{proposition} \label{prop:Z_0}
Let $B \bydef I_{\frac{nN(N+1)}{2}}-A^{[N]}(DF^{[N]}(\ba))$ and 
\begin{equation}
\label{eq:B_rescaled}
\tilde B \bydef  \cL^{[N]} B \left(\cL^{[N]}\right)^{-1}.
\end{equation}
Let the bound $\tilde Z_0 = (\tilde Z_0^{(1)}, \dots,\tilde Z_0^{(n)})$ defined component-wise by
\begin{equation*}
\tilde Z_0^{(i)} \bydef \sum_{j=1}^n K_{\tilde B}^{(i,j)}, \quad \forall~1\leq i\leq n.
\end{equation*}
Then
\begin{equation*}
\left\Vert \left(\left(I-\tilde A\tilde A^{\dag}\right)c\right)^{(i)}\right\Vert_{\ell_\nu^1} \leq \tilde Z_0^{(i)}, \quad \forall~1\leq i\leq n,
\end{equation*}
for all $c$ such that $\left\Vert c\right\Vert_{X}\leq 1$.
\end{proposition}
\begin{remark}
{\em
This bound can also be quite costly, because of the matrix-matrix multiplication required to get $B$. But again, once $B$ has been computed, we only need to do the multiplication by the diagonal matrices associated do $\cL^{[N]}$ and $\left(\cL^{[N]}\right)^{-1}$ to get $\tilde B$ and then to compute the quantities $K_{\tilde B}^{(i,j)}$ to get the new bound for any rescaling.
}
\end{remark}
\begin{proof}
We start by noticing that
\begin{equation*}
I-\tilde A\tilde A^{\dag} = \cL\left(I-AA^{\dag}\right)\cL^{-1}.
\end{equation*}
Then by definition of $A^{\dag}$ and $A$, $\left( \left(I-  A  A^{\dag}\right)c\right)_{\alpha}=0$ for all $\vert\alpha\vert \geq N$ and we have
\begin{equation*}
\left\Vert \left(\left(I- \tilde A \tilde A^{\dag}\right)c\right)^{(i)}\right\Vert_{\ell_\nu^1}=\left\Vert \left( \cL^{[N]} \left(I_\frac{nN(N+1)}{2} - D\left(DF^{[N]}(\ba)\right)\right)\left(\cL^{[N]}\right)^{-1} c^{[N]}\right)^{(i)} \right\Vert_{\ell_\nu^1},
\end{equation*}
and Lemma~\ref{lem:B^{[N]}} yields the formula for $\tilde Z_0$.
\end{proof}

\subsubsection{The bound \boldmath $\tilde Z_1$\unboldmath}
\label{sec:Z_1}

In this section we will need two additional results. The first one is a quantitative statement that $\ell_\nu^1$ is a Banach algebra and allows us to bound the nonlinear terms.
\begin{definition}
Let $u,v\in \ell_\nu^1$. We denote by $u\ast v$ the Cauchy product of $u$ and $v$, namely
\begin{equation*}
\left(u\ast v\right)_{\alpha} = \sum_{0\leq \beta\leq \alpha} u_{\alpha-\beta}v_{\beta},\qquad \forall~\vert\alpha\vert\geq 0,
\end{equation*}
where $\beta\leq\alpha$ means $\beta_i\leq\alpha_i$ for all $1\leq i\leq n_s$ and $(\alpha-\beta)_i=\alpha_i-\beta_i$ for all $1\leq i\leq n_s$.
\end{definition}
\begin{lemma}
\label{lem:convo}
\begin{equation*}
\forall~u,v\in \ell_\nu^1,\quad \left\Vert u\ast v\right\Vert_{\ell_\nu^1} \leq \left\Vert u\right\Vert_{\ell_\nu^1} \left\Vert v\right\Vert_{\ell_\nu^1}.
\end{equation*}
\end{lemma}
\noindent The second one bounds the action of the (infinite) diagonal part of $A$.
\begin{lemma}
\label{lem:A_tail}
Let $d\in X= \left(\ell_\nu^1\right)^n$, such that $d_{\alpha}=0$ for all $\vert\alpha\vert<N$. Then
\begin{equation*}
\left\Vert \left(Ad\right)^{(i)} \right\Vert_{\ell_\nu^1 } \leq \frac{1}{N\min\limits_{1\leq l\leq n_s}\left\vert \Re(\lambda_l) \right\vert} \left\Vert d^{(i)} \right\Vert_{\ell_\nu^1 }, \quad \forall~1\leq i\leq n.
\end{equation*}
\end{lemma} 
These two lemma allow us to get the $Z_1$ bound. 
\begin{proposition}
\label{prop:Z_1}
The bound $\tilde Z_1= \left( \tilde Z_1^{(1)} , \dots, \tilde Z_1^{(n)} \right)$ defined component-wise by
\begin{equation*}
\tilde Z_1^{(k)} = \frac{\sum_{1\leq i\leq n}\left\vert b_{\beta_i}^{(k)}\right\vert +\sum_{1\leq i,j \leq n} \left\vert b_{\beta_{i,j}}^{(k)}\right\vert \left\Vert \left(\cL (\ba)\right)^{(i)}\right\Vert_{\ell_\nu^1}}{N\min\limits_{1\leq i\leq n_s}\left\vert \Re(\lambda_i) \right\vert}, \quad \forall~1\leq k\leq n,
\end{equation*}
satisfies
\begin{equation*}
\left\Vert \left(\tilde A\left(D\tilde F(\cL(\ba))-\tilde A^{\dag}\right)c\right)^{(i)} \right\Vert_{\ell_\nu^1} \leq \tilde Z_1^{(i)}, \quad \forall~1\leq i\leq n,
\end{equation*}
for all $c$ such that $\left\Vert c\right\Vert_{X}\leq 1$. 
\end{proposition}
\begin{remark}
{\em
This bound is not costly, as we only need to get $\cL(\ba)$ from $\ba$ (a component-wise multiplication) and then to evaluate a finite sum to get the $\ell_\nu^1$ norm of $\cL(\ba)$.
}
\end{remark}
\begin{proof}
We first prove the bound without rescaling (that is for $\gamma=(1,\ldots,1)$). By definition of $A^{\dag}$, $\left(\left(DF(\ba)-A^{\dag}\right)c\right)_{\alpha}=0$ for all $\vert\alpha\vert <N$. For $\vert\alpha\vert \geq N$, remember that the general expression for $F$ is (for quadratic linearity)
\begin{equation*}
F_{\alpha}(a) = \left(\alpha_1\lambda_1 + \ldots + \alpha_{n_s}\lambda_{n_s}\right)a_{\alpha} - \sum_{\vert\beta\vert \leq 2} b_{\beta}\left( a^{\beta}\right)_{\alpha}, \quad\forall~\vert\alpha\vert\geq 2.
\end{equation*}
Then, again by definition of $A^{\dag}$, the $\left(\alpha_1\lambda_1 + \ldots + \alpha_{n_s}\lambda_{n_s}\right)$ term cancels out in $\left(\left(DF(\ba)-A^{\dag}\right)c\right)_{\alpha}$ and what is left is
\begin{equation}
\label{eq:DF-Adag}
\left(\left(DF(\ba)-A^{\dag}\right)c\right)_{\alpha} = -\left(\sum_{1\leq i\leq n}b_{\beta_i}c^{(i)}_{\alpha} +\sum_{1\leq i,j \leq n} b_{\beta_{i,j}}\left(\ba^{(i)}\ast c^{(j)}\right)_{\alpha}\right), \quad \forall~\vert\alpha\vert\geq N,
\end{equation}
where $\beta_i$ must be understood as the multi-index with $1$ at index $i$ and $0$ elsewhere, and $\beta_{i,j}$ as the multi-index with $1$ at indexes $i$ and $j$, and $0$ elsewhere. We then use Lemma~\ref{lem:convo} to get 
\begin{equation*}
\left\Vert \left(\left(DF(\ba)-A^{\dag}\right)c\right)^{(k)} \right\Vert_{\ell_\nu^1} \leq \sum_{1\leq i\leq n}\left\vert b_{\beta_i}^{(k)}\right\vert \left\Vert c^{(i)}\right\Vert_{\ell_\nu^1} +\sum_{1\leq i,j \leq n} \left\vert b_{\beta_{i,j}}^{(k)}\right\vert \left\Vert \ba^{(i)}\right\Vert_{\ell_\nu^1}\left\Vert c^{(j)}\right\Vert_{\ell_\nu^1}.
\end{equation*}
We now use Lemma~\ref{lem:A_tail} which yields
\begin{equation*}
\left\Vert \left(A\left(DF(\ba)-A^{\dag}\right)c\right)^{(k)} \right\Vert_{\ell_\nu^1} \leq \frac{\sum_{1\leq i\leq n}\left\vert b_{\beta_i}^{(k)}\right\vert \left\Vert c^{(i)}\right\Vert_{\ell_\nu^1} +\sum_{1\leq i,j \leq n} \left\vert b_{\beta_{i,j}}^{(k)}\right\vert \left\Vert \ba^{(i)}\right\Vert_{\ell_\nu^1}\left\Vert c^{(j)}\right\Vert_{\ell_\nu^1}}{N\min\limits_{1\leq l\leq n_s}\left\vert \Re(\lambda_l) \right\vert},
\end{equation*}
and the formula for $Z_1$ follows (in the particular case when $\gamma=(1,\ldots,1)$), since we assumed that $\left\Vert c\right\Vert_{X}\leq 1$. Now we want to get the general bound. First notice that
\begin{equation}
\label{eq:before_commuting}
\tilde A\left(D\tilde F(\cL(\ba)) - \tilde A^{\dag}\right)c = \cL A\left(DF(\ba)-A^{\dag}\right)\cL^{-1}c.
\end{equation}
Then, going back to \eqref{eq:DF-Adag} and using that $\ba\ast\cL^{-1}c = \cL^{-1}\left(\cL(\ba)\ast c\right)$, we get for all $\vert\alpha\vert\geq N$ that
\begin{equation}
\label{eq:DF-Adag_tilde}
\left(\left(DF(\ba)-A^{\dag}\right)\cL^{-1}c\right)_{\alpha} = -\left(\sum_{1\leq i\leq n}b_{\beta_i} \left(\cL^{-1}c\right)^{(i)}_{\alpha} +\sum_{1\leq i,j \leq n} b_{\beta_{i,j}}\left(\cL^{-1}\left(\left(\cL (\ba)\right)^{(i)}\ast c^{(j)}\right)\right)_{\alpha}\right).
\end{equation}
Then, since we only need to consider the action of the diagonal part of $A$ (that is for $\vert\alpha\vert\geq N$) we can commute $A$ and $\cL$ in~\eqref{eq:before_commuting}. Finally, applying $\cL$ to~\eqref{eq:DF-Adag_tilde} the $\cL$ and $\cL^{-1}$ cancel out and using again Lemma~\ref{lem:A_tail} we get the announced formula for $\tilde Z_1$.
\end{proof}

\subsubsection{ The bound \boldmath $\tilde Z_2$ \unboldmath}
\label{sec:Z_2}

To get the last bound we need a last lemma, which is a combination of Lemma~\ref{lem:B^{[N]}} and Lemma~\ref{lem:A_tail} and thus provides a bound on the full action of $A$.
\begin{lemma}
\label{lem:A}
For any $d\in \left(\ell_\nu^1\right)^n$ and for all $1\leq i\leq n$,
\begin{equation*}
\left\Vert (Ad)^{(i)} \right\Vert_{\ell_\nu^1} \leq \max\left(\frac{1}{N\min\limits_{1\leq l\leq n_s}\left\vert \Re(\lambda_l) \right\vert},K_{A^{[N]}}^{(i,i)}\right)\left\Vert d^{(i)} \right\Vert_{\ell_\nu^1} + \sum_{j\neq i} K_{A^{[N]}}^{(i,j)}\left\Vert d^{(j)} \right\Vert_{\ell_\nu^1}
\end{equation*}
\end{lemma}
\begin{proposition}
\label{prop:Z_2}
The bound $\tilde Z_2= \left( \tilde Z_2^{(1)} , \dots, \tilde Z_2^{(n)} \right)$ defined component-wise by
\begin{equation*}
\tilde Z_2^{(k)}=\max\left(\frac{1}{N\min\limits_{1\leq i\leq n_s}\left\vert \Re(\lambda_i) \right\vert},K_{\tilde A^{[N]}}^{(k,k)}\right) \sum_{1\leq i,j\leq n}\left\vert b_{\beta_{i,j}}^{(k)}\right\vert + \sum_{l\neq k} K_{\tilde A^{[N]}}^{(k,l)} \sum_{1\leq i,j\leq n}\left\vert b_{\beta_{i,j}}^{(l)}\right\vert, \quad \forall~1\leq k\leq n,
\end{equation*}
where
\begin{equation*}
\tilde A^{[N]} = \cL^{[N]} A^{[N]} \left(\cL^{[N]}\right)^{-1},
\end{equation*}
satisfies
\begin{equation*}
\left\Vert \left(\tilde AD^2\tilde F(\cL(\ba))(b,c)\right)^{(i)} \right\Vert_{\ell_\nu^1} \leq \tilde Z_2^{(i)},
\end{equation*}
for all $b$ and $c$ such that $\left\Vert b\right\Vert_{X}\leq 1$ and $\left\Vert c\right\Vert_{X}\leq 1$. 
\end{proposition}
\begin{remark}
{\em
The only costly part in this bound is to get $\tilde D$ (and the quantities $K_{\tilde D}^{(k,l)}$), but we already needed to compute $\tilde D$ for the $\tilde Y$ bound.
}
\end{remark}
\begin{proof}
Again we prove the bound without rescaling first (that is for $\gamma=(1,\ldots,1)$). Since we assume that $F$ is quadratic, we get that
\begin{equation}
\label{eq:D2F}
D^2F(\ba)(b,c) = -\sum_{1\leq i,j\leq n}b_{\beta_{i,j}}b^{(i)}\ast c^{(j)},
\end{equation}
with the same conventions as in Section~\ref{sec:Z_1} for the $\beta_{i,j}$. Therefore, using Lemma~\ref{lem:convo} and since $\left\Vert b\right\Vert_{X}\leq 1$ and $\left\Vert c\right\Vert_{X}\leq 1$,
\begin{equation*}
\left\Vert \left(D^2F(\ba)(b,c)\right)^{(k)}\right\Vert_{\ell_\nu^1} \leq \sum_{1\leq i,j\leq n}\left\vert b_{\beta_{i,j}}^{(k)}\right\vert.
\end{equation*}
Lemma~\ref{lem:A} then yields the formula for $Z_2$ (in the particular case when $\gamma=(1,\ldots,1)$). To get the general formula, we can compute
\begin{align*}
\tilde AD^2\tilde F(\cL(\ba))(b,c) &= \cL AD^2 F(\ba)(\cL^{-1}b,\cL^{-1}c) \\
& = \cL A \cL^{-1} D^2 F(\ba)(b,c),
\end{align*}
where we used $\left(\cL^{-1}b\right)\ast \left(\cL^{-1} c\right) = \cL^{-1}\left(b\ast c\right)$ in~\eqref{eq:D2F}. The infinite part of $\cL A \cL^{-1}$ (for $\vert\alpha\vert\geq N$) is the same as the one of $A$ since the infinite part of $A$ is diagonal. The only difference is that $\left(\cL A \cL^{-1}\right)^{[N]}=\cL^{[N]} D \left(\cL^{[N]}\right)^{-1}=\tilde D$, which yields the formula for $\tilde Z_2$.
\end{proof}

\subsection{Radii polynomials}

Let us sum up the results of the previous sections.
\begin{proposition}
Given $\gamma=(\gamma_1,\ldots,\gamma_{n_s})$, we consider $\tilde F$ defined as in \eqref{eq:F_tilde}. We also consider $\ba$ an element of $X=\left(\ell_\nu^1\right)^n$ such that $\left(\ba_{\alpha}\right)_{\alpha}=0$ for all $\vert\alpha\vert\geq N$ (in practice a numerical approximate zero of $F$) and the operator $\tilde T$ defined by~\eqref{eq:T_tilde}, \eqref{eq:to_make_JP_happy} and \eqref{eq:A}. Then the bound $\tilde Y$ defined in Proposition~\ref{prop:Y}, and the bound
\begin{equation*}
\tilde Z(r)=(\tilde Z_0+\tilde Z_1)r+\tilde Z_2 r^2,
\end{equation*}
where $\tilde Z_0$, $\tilde Z_1$ and $\tilde Z_2$ are defined in Propositions~\ref{prop:Z_0}, \ref{prop:Z_1} and \ref{prop:Z_2} respectively, satisfy the hypothesis~\eqref{eq:Y_and_Z} of Theorem~\ref{th:fixed_point}.
\end{proposition}
Then, for each $1\leq i\leq n$, $\tilde P^{(i)}$ defined in Theorem~\ref{th:fixed_point} is a quadratic polynomial. If there exists $r^*>0$ such that $\tilde P^{(i)}(r^*)<0$ for 
all $1\leq i \leq n$, then there exists an interval $\mathcal I=(r_0,r_1)$ such that $\tilde P^{(i)}(r)<0$ for all $1\leq i \leq n$ and for all $r \in \mathcal I$. By Theorem~\ref{th:fixed_point}, we know that for all $r\in \mathcal I$, within a ball of radius $r$ centered in $\cL(\ba)$ their exists a unique local parameterization of the manifold. Moreover, if one wants to make this fully rigorous, a final step consists of computing the bounds $\tilde Y$ and $\tilde Z$ with interval arithmetic and then check, still with interval arithmetic, that $\tilde P^{(i)}(r)$ is negative.

Finally, if the goal is to get a proof-valid parameterization while having the largest possible image, we process as follows. We start by computing the bounds (and the associated radii polynomials) without rescaling. Then if $\mathcal I$ is empty, or if $r_{max}<r_0$, we can rescale $\ba$ to $\cL(\ba)$ by some $\gamma$ and then compute the interval $\mathcal I$ associated to the rescaled polynomials $\tilde P^{(i)}$ (of course one should choose $\gamma_k<1$) but this time the computation of the coefficients of the polynomials, namely $\tilde Y$, $\tilde Z_0$, $\tilde Z_1$ and $\tilde Z_2$, are much faster thanks to the formulas of the previous sections. Conversely, if $r_0$ is small compared to $r_{max}$, we can rescale $\ba$ to $\cL(\ba)$ by some $\gamma$, this time with $\gamma_k>1$ larger and larger, which will give a larger and larger manifold patch associated to the rescaled parameterization, until we reach the limit of $r_0=r_{max}$. We explain more in detail how we do this on an example in Section~\ref{sec:bridge}.

\section{Examples}
\label{sec:examples}

\subsection{Defect-valid parameterizations for the Lorenz system} 
\label{sec:Lorenz}

As a first example, we consider the Lorenz system, given by the vector field
\begin{equation*}
g(x,y,z)=
\begin{pmatrix}
\sigma(y-x)\\
\rho x-y-xz\\
xy-\beta z
\end{pmatrix},
\end{equation*}
with standard parameter values : $\sigma=10$, $\beta=\frac{8}{3}$ and $\rho=28$. 
In this case it is well known that the origin has a two dimensional stable manifold. We detail on this example the method presented in Sections~\ref{sec:para_method} and \ref{sec:maximize} to automatically compute a maximal patch of the local stable manifold at $p=0$.

We start by recalling that the stable eigenvalues are
\begin{equation*}
\lambda_1=-\frac{1}{2}\left(\sigma+1+\sqrt{(\sigma-1)^2+4\sigma\rho}\right) \quad\text{and}\quad \lambda_2=-\beta ,
\end{equation*}
together with the stable eigenvectors
\begin{equation*}
V_1=\begin{pmatrix} \frac{\sigma}{\lambda_1+\sigma}\\ 1 \\ 0 \end{pmatrix}\quad\text{and}\quad V_2=\begin{pmatrix} 0 \\ 0 \\ 1 \end{pmatrix}.
\end{equation*}
As explained in Section~\ref{sec:para_method}, we look for a parameteriztion of the local stable manifold in the form of a power series $f$, which should satisfy the invariance equation
\begin{equation}
\label{eq:invariance_equation_Lorenz}
Df(\theta)\begin{pmatrix} \lambda_1 & 0\\ 0 & \lambda_2 \end{pmatrix} \theta = g(f(\theta)).
\end{equation}
together with the condition conditions
\begin{equation*}
f(0)=p \quad \text{and}\quad Df(0)=\begin{pmatrix} V_1 & V_2 \end{pmatrix}.
\end{equation*}
Notice that in this case the two stable eigenvalues are real and therefore we can directly work with a real power series defined on $[-1,1]^2$. Expanding $f$ into a power series, \eqref{eq:invariance_equation_Lorenz} rewrites as 
\begin{equation*}
\sum_{\vert \alpha \vert \geq 2} (\alpha_1 \lambda_1 + \alpha_2\lambda_2)a_{\alpha}\theta^{\alpha} = \sum_{\vert \alpha\vert \geq 2} 
\begin{pmatrix}
\sigma\left(a^{(2)}_{\alpha}-a^{(1)}_{\alpha}\right) \\
\rho a^{(1)}_{\alpha} - a^{(2)}_{\alpha} - \left(a^{(1)}\ast a^{(3)}\right)_{\alpha}\\
\left(a^{(1)}\ast a^{(2)}\right)_{\alpha} -\beta a^{(3)}_{\alpha}\\ 
\end{pmatrix}
\theta^{\alpha},
\end{equation*}
where
\begin{equation*}
a_{\alpha}=
\begin{pmatrix}
a^{(1)}_{\alpha} \\
a^{(2)}_{\alpha} \\
a^{(3)}_{\alpha}
\end{pmatrix}.
\end{equation*}
\noindent So we set $a_{0,0}=p$, $a_{1,0}=V_1$, $a_{0,1}=V_2$ and define $F=\left(F_{\alpha}\right)_{\vert\alpha\vert\geq 2}$, acting on $a=\left(a_{\alpha}\right)_{\vert\alpha\vert \geq 2}$, by
\begin{equation*}
F_{\alpha}(a)= (\alpha_1 \lambda_1 + \alpha_2\lambda_2)a_{\alpha} - 
\begin{pmatrix}
\sigma\left(a^{(2)}_{\alpha}-a^{(1)}_{\alpha}\right) \\
\rho a^{(1)}_{\alpha} - a^{(2)}_{\alpha} - \left(a^{(1)}\ast a^{(3)}\right)_{\alpha}\\
\left(a^{(1)}\ast a^{(2)}\right)_{\alpha} -\beta a^{(3)}_{\alpha}\\ 
\end{pmatrix},\quad \forall~\vert\alpha\vert \geq 2.
\end{equation*}
Our goal is now to find a numerical zero $\ba$ and then the rescaling $\gamma=(\gamma_1,\gamma_2)$ so that the parameterization $\tilde f$ defined as 
\begin{equation*}
\tilde f(\theta)=\sum_{\vert \alpha\vert \geq 0} \cL_{\alpha}(\ba)\theta^{\alpha}, \quad \forall~\theta\in [-1,1]^2,
\end{equation*} 
gives us the maximal patch of manifold, while checking (according to Definition~\ref{def:defect_valid}) that $\Vert \tilde F(\cL(\ba))\Vert_{X} < \varepsilon_{max}$, which will ensure that $\cL(\ba)$ is a good approximate parameterization.

First we fix an integer $N$ and consider a truncated version of $F$, that is 
\begin{equation*}
F^{[N]}=\left(F_{\alpha}\right)_{2\leq\vert\alpha\vert<N},
\end{equation*}
for which we can numerically compute a zero $\ba$ with Newton's method. Then we fix an $\varepsilon_{max}$ and use Method 1 described in Section~\ref{sec:maximize}. First we compute $F(\ba)$, which can be done explicitly because by construction $\ba_{\alpha}=0$ for any $\vert\alpha\vert\geq N$, so for $i=1$, $F^{(i)}(\ba)=0$ for any $\vert\alpha\vert\geq N$ and for $i\in\{2,3\}$, $F^{(2)}(\ba)=0$ for any $\vert\alpha\vert\geq 2N-1$ (because of the quadratic terms). 
Then we find numerically the curve in the plane $(\gamma_1,\gamma_2)$ that corresponds to $\Vert \tilde F(\cL(\ba))\Vert_{X}=\varepsilon_{max}$. In our case, we took a sample of values of $\gamma_1$ and for each we looked for the largest $\gamma_2$ for which $\Vert \tilde F(\cL(\ba))\Vert_{X}<\varepsilon_{max}$ (as explained in Section~\ref{sec:maximize} this doesn't require much computations since the coefficient of $F(\ba)$ are already known).
Finally we compute the surface of the corresponding patch of the manifold along this sample and find its maximum. The results are displayed in Figure~\ref{fig:lorenzFigB}, along with the results of similar 
computations for the unstable manifolds of the nontrivial equilibria, or ``eyes,'' of the attractor.

\begin{figure}[h!] 
\begin{center}
\subfigure[The curve of $(\gamma_1,\gamma_2)$ for which $\Vert \tilde F(\cL(\ba))\Vert_{X}=\varepsilon_{max}$ for the local stable manifold of the origin.]{\includegraphics[width=7.7cm]{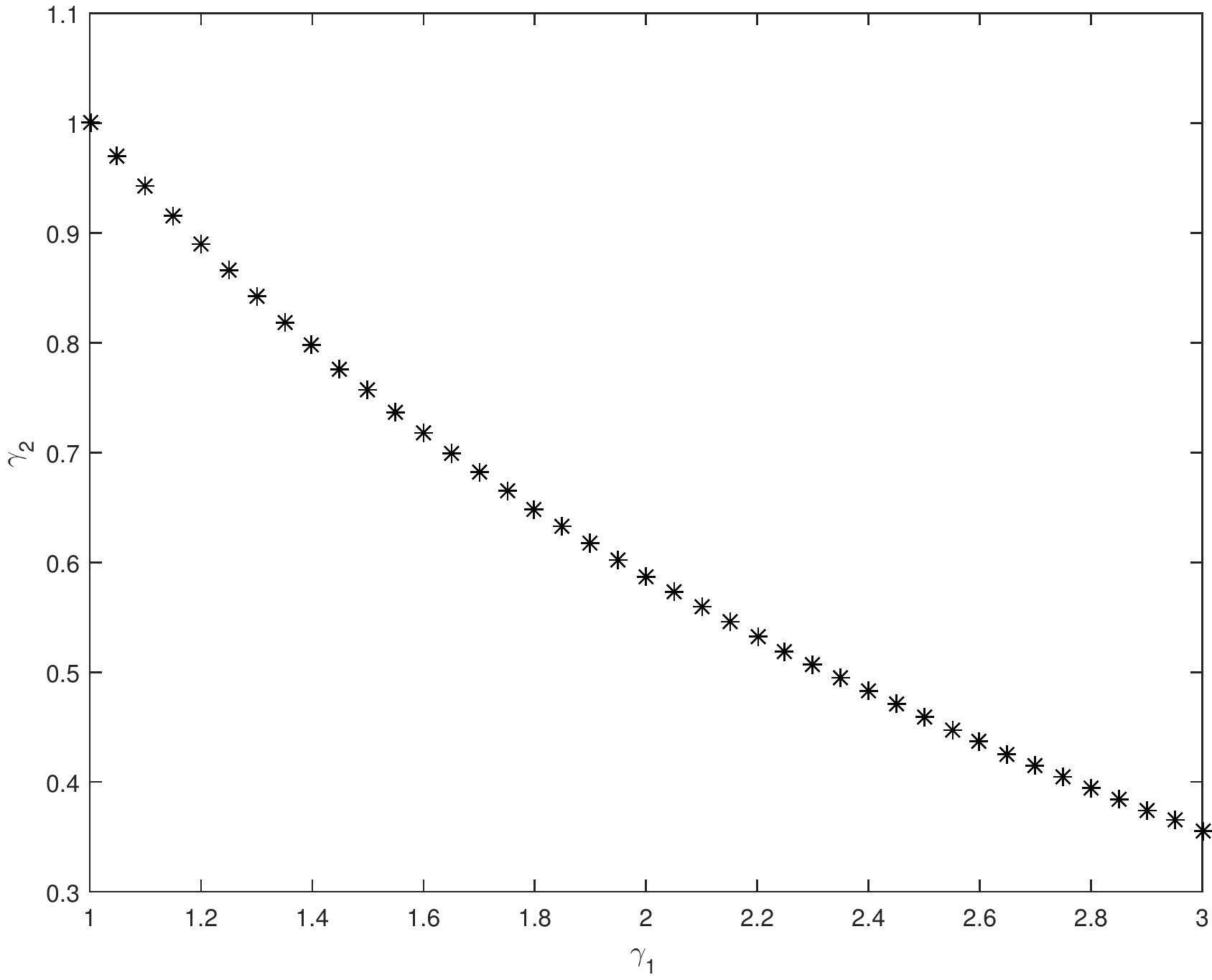}}
\subfigure[The corresponding values of the surface area (again for the local stable manifold of the origin).]{\includegraphics[width=7.7cm]{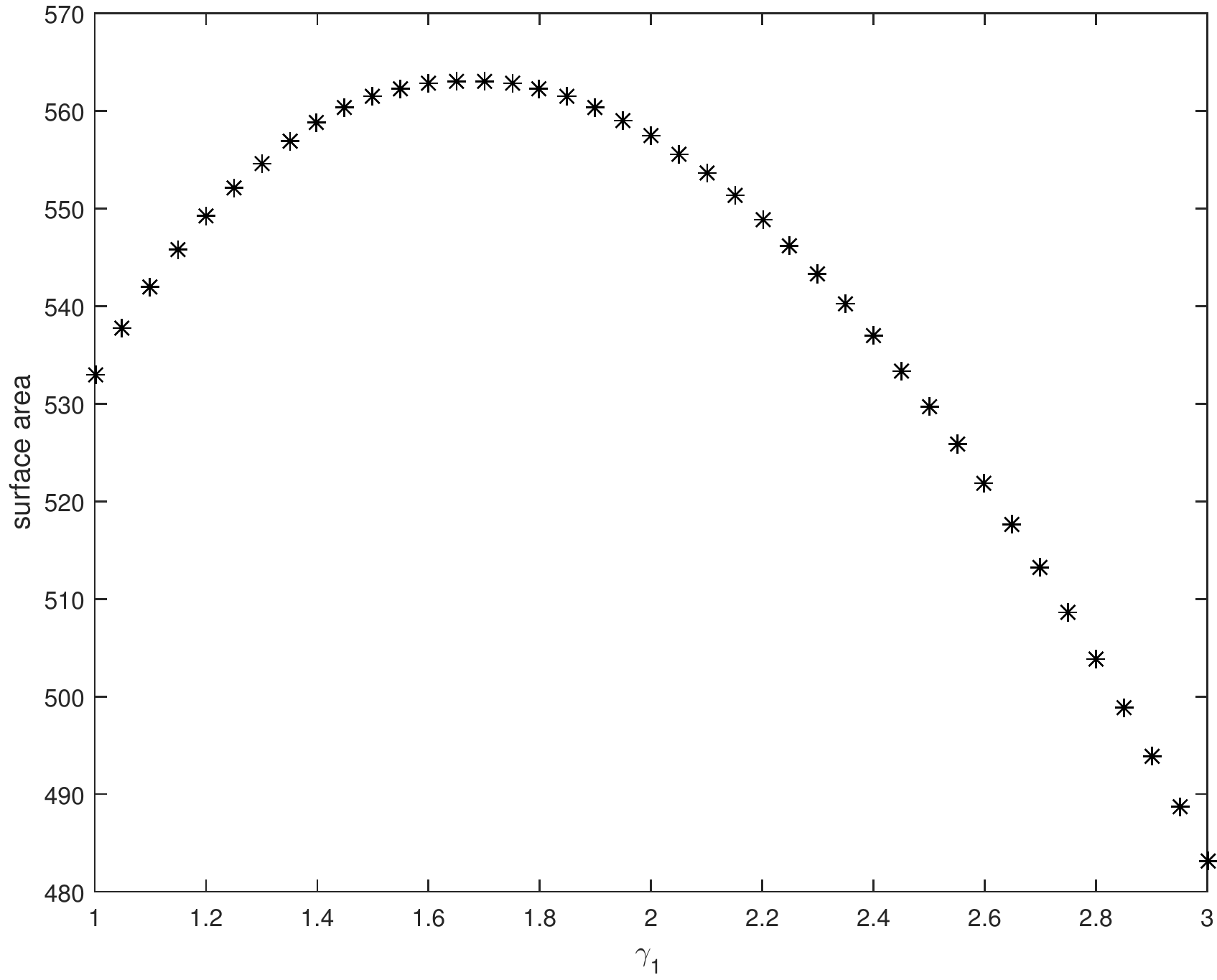}}\\
\subfigure[Lorenz System: local stable manifold of the origin (with the rescaling maximizing the surface area, i.e. $\gamma_1=1.7$ and $\gamma_2=0.68$) and 
local unstable manifolds of the eyes.  The unstable manifolds
 have complex conjugate eigenvalues, so we simply maximize the length
 of the eigenvectors.]{\includegraphics[width=12cm]{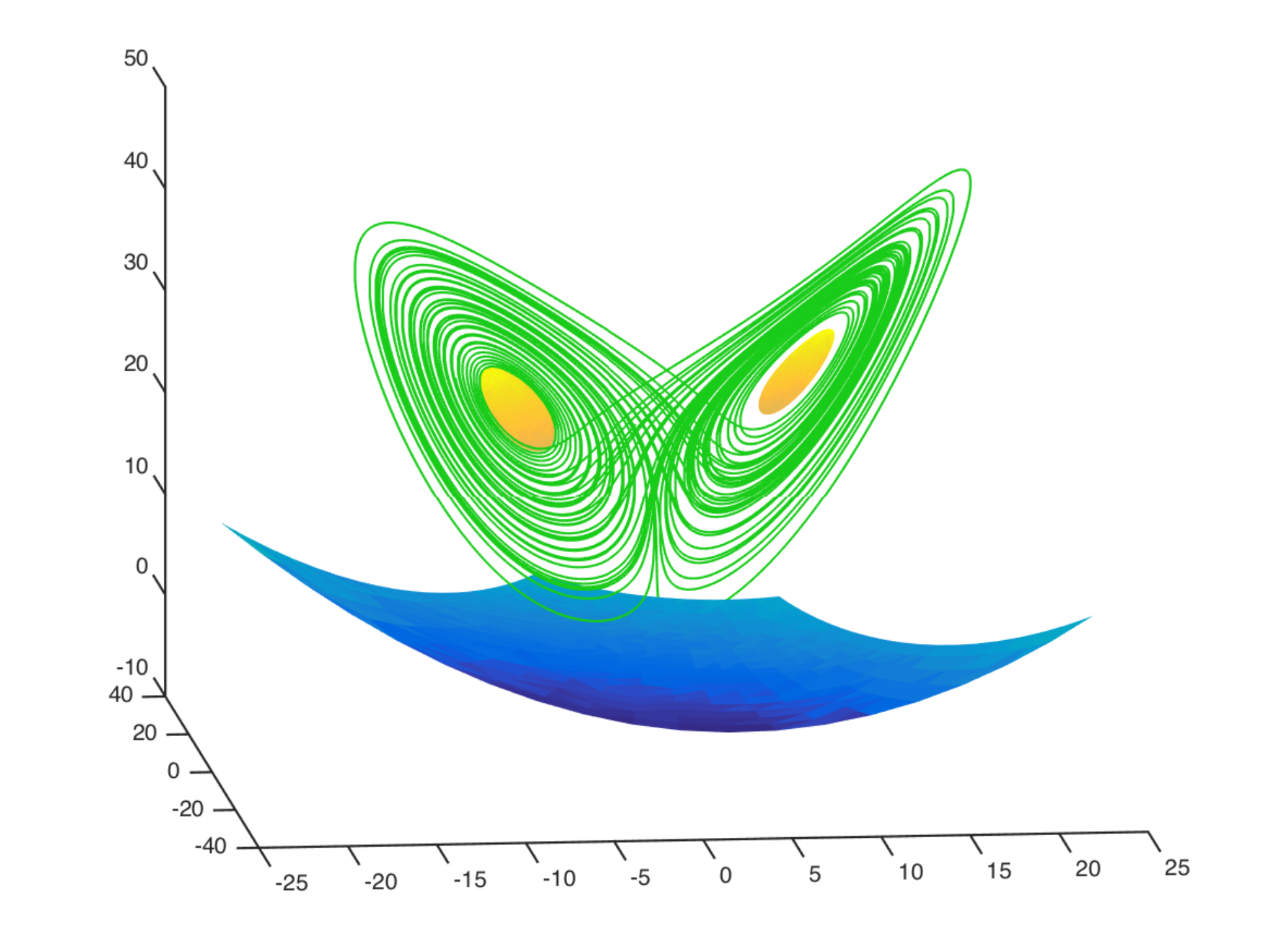}}
\end{center}
%\vspace{-.3cm}
\caption{For each manifold we take a defect constraint of $\varepsilon_{max}=10^{-5}$. The order of the parameterizations is $N=50$ for the eyes and $N=30$ for the stable local manifold of the origin.
} \label{fig:lorenzFigB}
\end{figure}

By way of contrast we 
consider another parameterization of the local stable manifold at $p$ but focusing on the slow direction given by $V_2$. Therefore we apply Method 2 described in Section~\ref{sec:maximize}: we define the ratio $\varrho\bydef \left\vert \frac{\lambda_1}{\lambda_2}\right\vert$ and only consider rescalings of the form $\gamma=(\gamma_1,\varrho\gamma_1)$. Then we simply find numerically the largest $\gamma_1$ such that the rescaled parameterization $\cL(\ba)$ is defect valid, and obtain the results displayed in Figure~\ref{fig:fastSlow_Lorenz}.

\begin{figure}[h!] 
\begin{center}
\includegraphics[width=12cm]{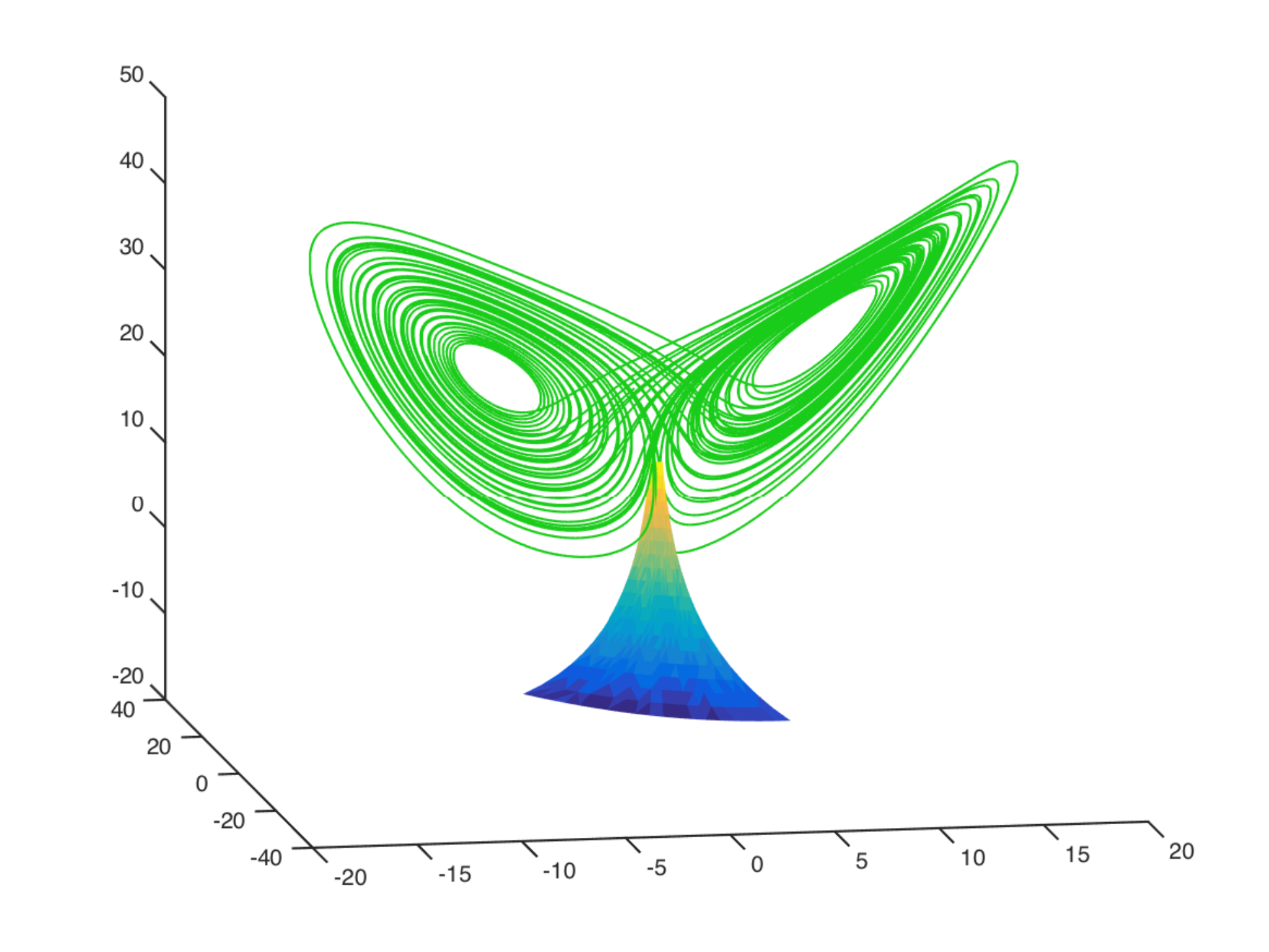}
\end{center}
%\vspace{-.3cm}
\caption{
Lorenz System: the figure illustrates the results of maximizing 
the lengths of the stable eigenvectors of the origin subject to 
the constraint that the slow eigenvector is $ \varrho = |\lambda_1|/|\lambda_2|$ times
longer than the fast eigenvector and that the defect is less
than $\varepsilon_{max}=10^{-5}$. The order of this parameterization is $N=50$.
Note that the resulting patch covers more of the slow stable manifold than 
the patch shown in Figure \ref{fig:lorenzFigB}, however the surface area is smaller.
} \label{fig:fastSlow_Lorenz}
\end{figure}

\subsection{Defect-valid parameterizations for the FitzHugh-Nagumo equations} 
\label{sec:FHN}

We consider the vector field given by 
\begin{equation*}
g(u,v,w)=
\begin{pmatrix}
v\\
\frac{1}{\Delta}\left(sv+w-q+u^3-(1+\sigma)u^2+\sigma u\right)\\
\frac{\varepsilon}{s}\left(u-\zeta w\right)
\end{pmatrix},
\end{equation*}
where
\begin{equation*}
\sigma=\frac{1}{10},\ s=1.37,\ \Delta=1,\ q=0.001,\ \varepsilon=0.15 \text{ and } \zeta=5.
\end{equation*}
There are trivial zeros given by $v=0$, $w=\frac{u}{\zeta}$ and $u$ solution of the cubic equation
\begin{equation*}
u^3-(1+\sigma)u^2+\left(\sigma+\frac{1}{\zeta} \right)u-q=0.
\end{equation*}
We want to compute the stable local manifold at one of them:
\begin{equation*}
p\simeq
\begin{pmatrix}
0.003374970076610\\
0\\
0.000674994015322
\end{pmatrix}.
\end{equation*}
With the selected values of the parameters we have two real stable eigenvalues at this point $p$:
\begin{equation*}
\lambda_1\simeq -0.662724919921474 \quad \text{and}\quad \lambda_2\simeq -0.184083645070452, 
\end{equation*} 
with associated eigenvectors
\begin{equation*}
V_1\simeq\begin{pmatrix}
  -0.576099055982516 \\
   0.381795200742850 \\
  -0.722732524787547 
\end{pmatrix}
\quad \text{and} \quad 
V_2\simeq\begin{pmatrix}
  -0.966141520359494 \\
   0.177850852721684 \\
  -0.186921472344981
\end{pmatrix}.
\end{equation*}
In this case we also want to compute a parameterization of the local stable manifold at $p$ focusing more on the slow direction given by $V_2$. Therefore we again apply Method 2 and obtain the results displayed in Figure~\ref{fig:fitzHugh_fastSlow}.

\begin{figure}[h!] 
\begin{center}
\includegraphics[width=6cm]{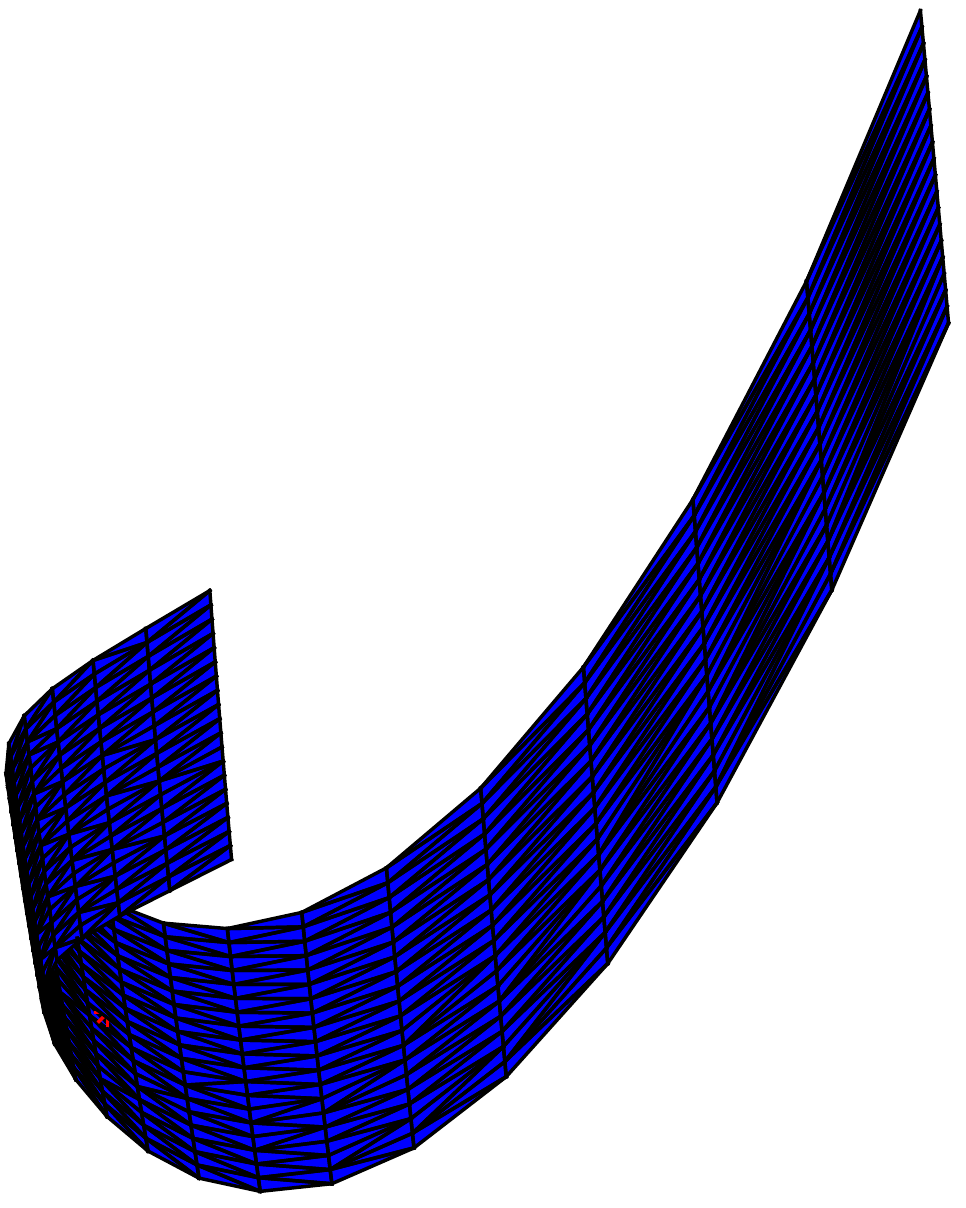}
\end{center}
\vspace{-.3cm}
\caption{
FitzHugh-Nagumo System: the figure illustrates the results of maximizing 
the lengths of the stable eigenvectors of the origin subject to 
the constraint that the slow eigenvector is $\varrho= |\lambda_1|/|\lambda_2|$ times
longer than the fast eigenvector and that the defect is less
than $10^{-5}$. The order of this parameterization is $N=30$. The red star indicates the location of the equilibrium.
The local manifold illustrated here is not the graph of any function over 
the stable eigenspace, i.e. the parameterization follows a fold in the 
manifold. Note that the triangulation in the figure is an artifact of the  
plotting procedure for the manifolds, and not part of the parameterization computation.
If a finer mesh is desired we simply evaluate the polynomial approximation 
at more points.  It is not necessary to re-compute the parameterization.
} \label{fig:fitzHugh_fastSlow}
\end{figure}

\subsection{Proof-valid parameterizations for the suspension bridge equation}  \label{sec:bridge}

We consider the vector field
\begin{equation*}
g(v)=\begin{pmatrix} v_2+v_1v_2\\ v_3\\ v_4\\-\beta v_3-v_1\end{pmatrix},
\end{equation*}
which is obtained after a change of variable when one looks for travelling waves in the suspension bridge equation (e.g. see \cite{MR1448828,MR2220064})
\begin{equation*}
\frac{\partial^2 u}{\partial t^2} = -\frac{\partial^4 u}{\partial x^4} -\left(e^u -1\right).
\end{equation*}

We are going to rigorously compute the stable manifold at 0 (for a given $\beta\in(0,2)$), which is two-dimensional. The stable eigenvalues are $\lambda$ and $\overline\lambda$, where
\begin{equation}
\label{eq:lambda}
\lambda = -\frac{1}{2}\sqrt{2-\beta} +i\frac{1}{2}\sqrt{2+\beta},
\end{equation}
and associated eigenvectors are given by
\begin{equation*}
V_1=\begin{pmatrix}
1\\ \lambda \\ \lambda^2 \\ \lambda^3
\end{pmatrix} \quad 
\text{and} \quad V_2=\overline V_1.
\end{equation*}
We then define $F=\left(F_{\alpha}\right)_{\vert\alpha\vert\geq 0}$, acting on $a=(a_{\alpha})_{\vert\alpha\vert\geq 0}\in \left(\ell^1\right)^4$, by
\begin{equation*}
F_{\alpha}(a) =
\left\{
\begin{aligned}
& a_0 - 0, \quad &\text{if }\alpha=0, \\
& a_{1,0} - V_1,\quad &\text{if }\alpha = (1,0) \\
& a_{0,1} - V_2,\quad &\text{if }\alpha = (0,1) \\
& (\alpha_1 \lambda + \alpha_2\overline\lambda)a_{\alpha} - 
\begin{pmatrix}
a_{\alpha}^{(2)} + (a^{(1)}\ast a^{(2)})_{\alpha} \\
a_{\alpha}^{(3)} \\
a_{\alpha}^{(4)} \\
-a_{\alpha}^{(1)} -\beta a_{\alpha}^{(3)}
\end{pmatrix}, \quad&\forall~\vert\alpha\vert\geq 2.
\end{aligned}
\right.
\end{equation*}
This time since the eigenvalues are not real, we consider complex parameterization $a$, i.e. $a_{\alpha}\in \C^4$ for all $\alpha$. Then we compute a numerical zero $\ba$ with the method described in Section~\ref{sec:approximation}. To rigorously prove the existence of a nearby solution $a$ we follow the ideas exposed in Section~\ref{sec:maximize} and consider an operator $T$ of the form 
\begin{equation*}
T:a\mapsto a-AF(a).
\end{equation*}
The following infinite matrix should be a good approximation of $DF(\ba)$ (at least for $N$ large enough)
\begin{equation*}
A^{\dag}=
\begin{pmatrix}
DF^{[N]}(\ba) & & 0 & \\
 & A_N & & \\
0 & & A_{N+1} & \\
 & & & \ddots\\
\end{pmatrix},
\end{equation*}
where for each $k\geq N$, $A_k$ is a $4(k+1)$ by $4(k+1)$ bloc diagonal matrix defined as
\begin{equation*}
A_k=
\begin{pmatrix}
k\lambda I_4 & & 0 & \\
 & ((k-1)\lambda+\overline\lambda)I_4 & & \\
0 & & \ddots & \\
 & & & k\overline\lambda I_4
\end{pmatrix},
\end{equation*}
with $I_4$ the 4 by 4 identity matrix. Therefore we define
\begin{equation*}
A =
\begin{pmatrix}
D & & 0 & \\
 & M_N & & \\
0 & & M_{N+1} & \\
 & & & \ddots\\
\end{pmatrix},
\end{equation*}
where $D$ is a numerical approximation of $DF^{[N]}(\ba)^{-1}$ while the $M_k=A_k^{-1}$ are exact inverses.

We are now ready to compute the bounds $\tilde Y$ and $\tilde Z$ defined in Section~\ref{sec:maximize} in order to apply Theorem~\ref{th:fixed_point} an prove the existence of a true parameterization near $\ba$. In practice, we first compute the bounds without rescaling (that is for $\gamma=(1,1)$) and denote them simply $Y$ and $Z$, and then we find the largest rescaling for which the parameterization is still proof valid.

\subsubsection{Computation of the bounds \boldmath $Y$ and $Z$\unboldmath, and of the radii polynomials}

Concerning the bounds $Y$ and $Z_0$, there is nothing to add or to specify to what was said in Section~\ref{sec:rigorous}. We set, for $1\leq i\leq 4$
\begin{equation*}
Y^{(i)}=\left\Vert \left(AF(\ba)\right)^{(i)}\right\Vert_{\ell^{1}_{\nu}},
\end{equation*}
and 
\begin{equation*}
Z_0^{(i)}=\sum_{j=1}^4 K_{B}^{(i,j)},
\end{equation*}
where the $K_{B}^{(i,j)}$ are defined as in Section~\ref{sec:Z_0}. For $Z_1$ and $Z_2$ we can specify the bounds of Section~\ref{sec:rigorous}, because we now work with a specific non linearity. We get
\begin{equation*}
Z^{(1)}_1 = \frac{1 + \left\Vert\ba^{(1)} \right\Vert_{\ell^{1}_{\nu}} + \left\Vert\ba^{(2)} \right\Vert_{\ell^{1}_{\nu}}}{N\vert\Re(\lambda)\vert},\quad Z^{(2)}_1 = \frac{1}{N\vert\Re(\lambda)\vert},\quad Z^{(3)}_1 = \frac{1}{N\vert\Re(\lambda)\vert},\quad Z^{(4)}_1 = \frac{1+\beta}{N\vert\Re(\lambda)\vert},
\end{equation*}
and
\begin{equation*}
Z^{(1)}_2 =2\max\left( K^{(1,1)}_D,\frac{1}{\vert\Re(\lambda)\vert N}\right),\quad Z^{(2)}_2 = 2K^{(2,1)}_D,\quad Z^{(3)}_2 = 2K^{(3,1)}_D,\quad Z^{(4)}_2 = 2K^{(4,1)}_D.
\end{equation*}
Now we can consider, for all $1\leq i\leq 4$, the radii polynomial defined by
\begin{equation*}
P^{(i)}(r) = Y^{(i)} + (Z^{(i)}_0+Z^{(i)}_1-1)r + Z^{(i)}_2r^2
\end{equation*}
and we can try and look for a positive $r$ such that $P^{(i)}(r)<0$ for all $1\leq i \leq 4$.

\begin{remark}
{\em
$Y$ should be very small if $\ba$ is a good approximative zero of $F$. $Z_0$ should also be very small because $B=I_{\frac{N(N+1)}{2}}-D(DF^{[N]}(\ba))$ and $D$ is a numerical inverse of $(DF^{[N]}(\ba))$. Finally $Z_1$ can be made very small by choosing $N$ large enough. Therefore the radii polynomials are of the form
\begin{equation*}
 P^{(i)} (r) = \epsilon - (1-\eta)r + Z^{(i)}_2r^2
\end{equation*} 
where $\epsilon$ could be made arbitrarily small if we could get an arbitrarily good approximation $\ba$ and $\eta$ could be made arbitrarily small if we could take with an arbitrarily large $N$ (and if we could numerically compute inverses of matrices with sufficient accuracy). So up to having sufficient \textit{computational precision} we should always be able to find a positive $r$ such that $P^{(i)}(r)<0$.
}
\end{remark}

\subsubsection{Results}

For this problem we are interested in proving (rigorously and with and error bound $r$ smaller than $r_{max}$) the largest possible patch of the stable manifold, for values of $\beta$ between 0.5 and 2. Since we already computed the bounds $Y$, $Z_0$, $Z_1$ and $Z_2$ without rescaling, we can now easily compute the radii polynomial $\tilde P$ for any rescaling, and so we look by dichotomy for the largest $\gamma$ such that the rescaled radii polynomial $\tilde P$ has a positive root $r_0$ which is less or equal to $r_{max}$. Notice that the eigenvalues are complex conjugated for this problem and that is why we only consider uniform rescaling (i.e. $\gamma_1=\gamma_2$).

When $\beta$ goes to 2, the real part of $\lambda$ goes to 0 (remember \eqref{eq:lambda}) so we expect it to be harder and harder to compute the manifold when $\beta$ goes to 2. Indeed we observe that the largest $\gamma$ for which we are able to do the proof becomes smaller and smaller when $\beta$ goes to 2 (see Figure~\ref{fig:gamma_vs_beta}). The computations were made with $N=30$, $\nu=1$ and $r_{max}=10^{-5}$ for the proof.

\begin{figure}[htbp]
\begin{center}
\includegraphics[width=8cm]{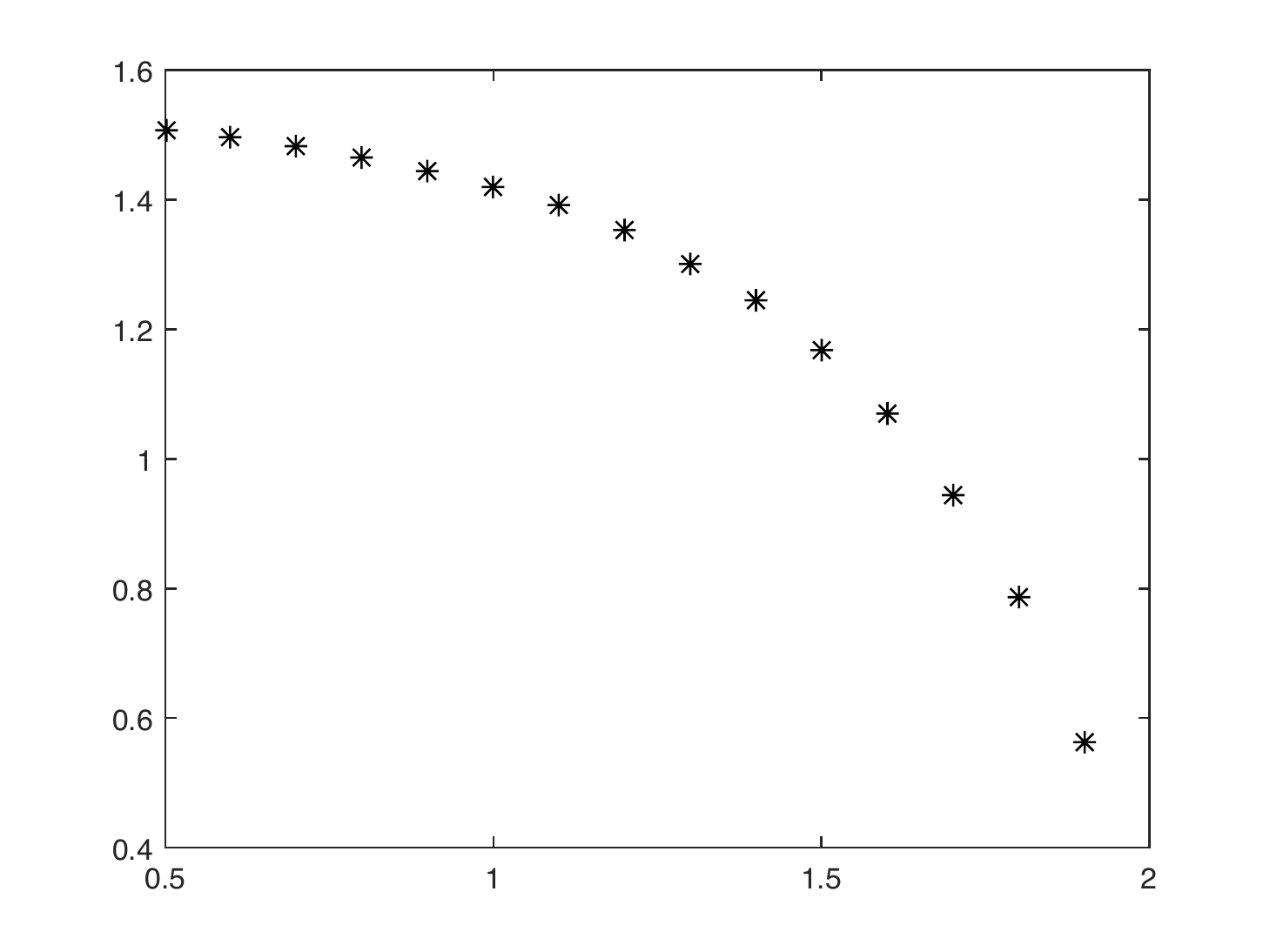}
\caption{Maximal value of $\gamma$ for which we can still do the proof with $r\leq r_{max}$, 
for different values of $\beta$. The manifold computations are completely automated.}
\label{fig:gamma_vs_beta}
\end{center}
\end{figure}

\begin{remark}
{\em
Another interesting point here is that a closer look at the bound $Z_0$ shows why it is better to take $\nu=1$. The matrix $B$ is supposed to be approximatively 0, and we want the terms $K_{B}^{(i,j)}$ of Lemma~\ref{lem:B^{[N]}} to be as small as possible, but their definition
\begin{equation*}
K_{B}^{(i,j)} = \max\limits_{\vert\beta\vert<N} \left(\frac{1}{\nu^{\vert\beta\vert}} \sum_{\vert\alpha\vert<N}  \left\vert B^{(i,j)}_{\alpha,\beta}\right\vert \nu^{\vert\alpha\vert}\right).
\end{equation*}
show that there is a risk of numerical errors if $\nu$ is too small or too large, hence our choice of always considering $\nu=1$. 

To speed up the process of redoing the proof after a rescaling, we kept track of the $\gamma$ dependency in the bound $Y$ and $Z$, and constructed the rescaled bound $\tilde Y$ and $\tilde Z$ based on the original ones. However by doing things this way we introduce in the $\tilde Z_0$ bound the same kind of instability that comes with taking $\nu\neq 1$ (see \eqref{eq:B_rescaled}). If the $\tilde Z_0$ bound becomes too big, we could deal with it (at the expense of speed), by recomputing all the bounds without using the fact that they came from a rescaling and thus eliminating this numerical instability issue.
}
\end{remark}

\section{Acknowledgments}

The first author was partially supported by the ANR-13-BS01-0004 funded by the
French Ministry of Research.
The second author was supported by an NSERC discovery grant.
The third author was partially supported by National Science Foundation grant DMS 1318172.

\bibliographystyle{unsrt}
\bibliography{papers}

%\bibitem{webpage}
%{M. Breden, J.-P. Lessard, and J.D. Mireles James},
%\newblock  MATLAB codes available at  
%\newblock {\tt http://archimede.mat.ulaval.ca/jplessard/MaximizingManifold/}

\end{document}